\documentclass[12pt,reqno]{amsart}
\usepackage{amsmath} 

\usepackage{amsthm}
\usepackage{amssymb}
\usepackage{graphics}
\usepackage{latexsym}

\numberwithin{equation}{section}
\newtheorem{thm}{Theorem}[section]
\newtheorem{prop}[thm]{Proposition}
\newtheorem{lem}[thm]{Lemma}
\newtheorem{cor}[thm]{Corollary}

\theoremstyle{remark}
\newtheorem{rem}{Remark}[section]
\newtheorem{defn}{Definition}
 
\newcommand{\laplacian}{\Delta}
\newcommand{\BBB}{\mathbb}
\newcommand{\R}{{\BBB R}}
\newcommand{\Z}{{\BBB Z}}

\newcommand{\C}{{\BBB C}}
\newcommand{\LR}[1]{{\langle {#1} \rangle }}
\newcommand{\lec}{{\ \lesssim \ }}
\newcommand{\gec}{{\ \gtrsim \ }}

\newcommand{\cross}{\times}

\newcommand{\e}{\varepsilon}
\newcommand{\ta}{\tau}

\newcommand{\p}{\partial}
\newcommand{\la}{\lambda}

\newcommand{\de}{\delta}
\newcommand{\om}{\omega}

\newcommand{\na}{\nabla}

\newcommand{\supp}{\operatorname{supp}}

\newcommand{\I}{\infty}

\newcommand{\EQ}[1]{\begin{equation} \begin{split} #1
 \end{split} \end{equation}}
\newcommand{\EQS}[1]{\begin{align} #1 \end{align}}
\newcommand{\EQQS}[1]{\begin{align*} #1 \end{align*}}
\newcommand{\EQQ}[1]{\begin{equation*} \begin{split} #1
 \end{split} \end{equation*}}

\newcommand{\ol}{\overline}
\newcommand{\ds}{\displaystyle}
\newcommand{\sub}{\subset}

\newcommand{\F}{\mathcal{F}}
\newcommand{\1}{{\mathbf 1}}

\newcommand{\ha}{\widehat}
\newcommand{\til}{\tilde}

\title[Well-posedness of the Klein-Gordon-Zakharov system]{Well-posedness
for the Cauchy problem \\
of the Klein-Gordon-Zakharov system \\
in five and more dimensions
}

\author[I.  Kato]{Isao Kato}
\author[S. Kinoshita]{Shinya Kinoshita}
\address[Isao Kato]{Graduate School of Mathematics, Nagoya University,
Chikusa-ku, Nagoya, 464-8602, Japan}
\address[Shinya Kinoshita]{Graduate School of Mathematics, Nagoya University,
Chikusa-ku, Nagoya, 464-8602, Japan}
\email[Isao Kato]{kato.isao@f.mbox.nagoya-u.ac.jp}
\email[Shinya Kinoshita]{m12018b@math.nagoya-u.ac.jp}


\begin{document}

\begin{abstract}
We study the Cauchy problem of the Klein-Gordon-Zakharov system in spatial dimension $d \ge 5$ with initial datum 
$(u, \partial_t u, n, \partial_t n)|_{t=0} \in H^{s+1}(\mathbb{R}^d) \times H^s(\mathbb{R}^d) \times \dot{H}^s(\mathbb{R}^d) \times \dot{H}^{s-1}(\mathbb{R}^d)$. The critical value of $s$ is $s_c=d/2-2$. 
By $U^2, V^2$ type spaces, we prove that the small data global well-posedness and scattering hold at $s=s_c$ in $d \ge 5$.   
\end{abstract}
\maketitle
\setcounter{page}{001}


\section{Introduction}
We consider the Cauchy problem of the Klein-Gordon-Zakharov system: 
\EQS{
 \begin{cases}
  (\p_t^2  - \laplacian  + 1)u = -nu, \qquad (t,x) \in [-T,T] \cross \R^d, \\
  (\p_t^2 - c^2 \laplacian )n = \laplacian |u|^2, \qquad \ \, (t,x) \in [-T,T] \cross \R^d, \\
  (u, \p_t u, n, \p_t n)|_{t=0} = (u_0, u_1, n_0, n_1) \\
   \qquad  \qquad  \qquad  \qquad \in H^{s+1}(\R^d) \cross H^s(\R^d) \cross \dot{H}^s(\R^d) \cross \dot{H}^{s-1}(\R^d), 
                                                                                                                                            \label{KGZ}
 \end{cases}  
}
where $u, n$ are real valued functions, $d \ge 5, c > 0$ and $c \neq 1$. \eqref{KGZ} describes the interaction of the Langmuir wave and the ion acoustic wave in a plasma. Physically, $c$ satisfies $0 < c < 1$. When $d=3$, Ozawa, Tsutaya and Tsutsumi ~\cite{OTT2}  proved that \eqref{KGZ} is globally  well-posed in the energy space $H^1(\R^3) \cross L^2(\R^3) \cross L^2(\R^3) \cross \dot{H}^{-1}(\R^3)$. They applied the Fourier restriction norm method to obtain the local well-posedness. Then by the local well-posedness and the energy method, they obtained the global well-posedness. For $d=3$, Guo, Nakanishi and Wang ~\cite{GNW} proved the scattering in the energy class with small, radial initial data. They applied the normal form reduction and the radial Strichartz estimates. If we transform $u_{\pm} := \om_1 u \pm i\p_t u, n_{\pm} := n \pm i(c\om)^{-1}\p_t n, \om_1:=(1-\laplacian )^{1/2}, \om := (-\laplacian )^{1/2}$, then \eqref{KGZ} is equivalent to the following. 
\EQS{
 \begin{cases}
  (i\p_t \mp \om_1) u_{\pm} 
    = \pm (1/4)(n_+ + n_-)(\om_1^{-1}u_+ + \om_1^{-1}u_-), 
               \quad (t,x) \in [-T,T] \cross \R^d, \\
  (i\p_t \mp c\om )n_{\pm} 
    = \pm (4c)^{-1}\om | \om_1^{-1} u_+ + \om_1^{-1} u_-|^2, \qquad \qquad (t,x) \in [-T,T] \cross \R^d, \\
  (u_{\pm}, n_{\pm})|_{t=0} = (u_{\pm 0}, n_{\pm 0}) 
                                \in H^s(\R^d) \cross \dot{H}^s(\R^d). 
                                                                                                                                            \label{KGZ'}
 \end{cases}  
}
Our main result is as follows. 
\begin{thm}  \label{mth}
Let $d \ge 5, s=s_c = d/2-2$ and assume the initial data $(u_{\pm 0},n_{\pm 0}) \in H^s(\R^d) \cross \dot{H}^s(\R^d)$ is small. Then, \eqref{KGZ'} is globally well-posed in $H^s(\R^d) \cross \dot{H}^s(\R^d)$.   
\end{thm}
\begin{cor} \label{Scatter}
The solution obtained in Theorem \ref{mth} scatters as $t \to \pm \I$. 
\end{cor}
For more precise statement of Theorem \ref{mth} and Corollary \ref{Scatter}, see Propositions \ref{main_prop1}, \ref{main_prop2}. ~\cite{Ka} considered \eqref{KGZ'} for $d \ge 4, 0 <c$ and $c \neq 1$. ~\cite{Ka} applied $U^2, V^2$ type spaces and obtained \eqref{KGZ'} is globally well-posed in $H^{s_c}(\R^d) \cross \dot{H}^{s_c}(\R^d)$ if the initial data is small and radial. $U^2, V^2$ type spaces were introduced by Koch and Tataru~\cite{KT1}. These spaces works well as one consider well-posedness at the critical space~\cite{HHK},~\cite{Hir},~\cite{IKO},~\cite{KaT}. Theorem \ref{mth} is proved by the Banach fixed point theorem. The key is the bilinear estimate (Proposition \ref{BE}). For $d \ge 5$, it seemed difficult to prove Proposition \ref{BE} only by applying $U^2, V^2$ type spaces, the modulation estimate (Proposition \ref{modulation}, Lemma \ref{Recovery}) and the Strichartz type estimates (Proposition \ref{Str}) for a nonlinear interaction~\cite{Ka}.  In the present paper, to overcome the difficulty, we derive the bilinear Strichartz estimate for the nonlinear interaction and then we are able to prove Proposition \ref{BE}. See Proposition \ref{bi-str} for the bilinear Strichartz estimate. $c \neq 1$ plays an important role in the proof of the bilinear Strichartz estimate as well as in the proof of Lemma \ref{Recovery}.

In Section $2$, we prepare some notations and lemmas with respect to $U^p, V^p$, in Section $3$, we prove the bilinear estimates and in Section $4$, we prove the main result.  

\section*{Acknowledgement}
The authors appreciate Professor M. Sugimoto and Professor K. Tsugawa for giving many useful advices to the authors. The second author is supported by Grant-in-Aid for JSPS Research Fellow 16J11453.

\section{Notations and Preliminary Lemmas}
In this section, we prepare some lemmas, propositions and notations to prove the main theorem. 
$A\lec B$ means that there exists $C>0$ such that $A \le CB.$ 
Also, $A\sim B$ means $A\lec B$ and $B\lec A.$  
Let $u=u(t,x).\ \F_t u,\ \F_x u$ denote the Fourier transform of $u$ in time, space, respectively. 
$\F_{t,\, x} u = \F u = \ha{u}$ denotes the Fourier transform of $u$ in space and time.    
Let $\mathcal{Z}$ be the set of finite partitions $-\I = t_0<t_1<...<t_K = \I$ and let $\mathcal{Z}_0$ 
be the set of finite partitions $-\I<t_0<t_1<...<t_K \le \I$. 
\begin{defn}
Let $1\le p< \I$.   
For $\{t_k\}_{k=0}^K \in \mathcal{Z}$ and $\{ \phi_k\}_{k=0}^{K-1}\subset L^2_x$ with   
$\sum_{k=0}^{K-1} \|\phi_k \|_{L^2_x}^p=1$, we call the function $a : \R \to L^2_x$ given by 
\EQQS{
 a=\sum_{k=1}^K \1_{[t_{k-1},\, t_k)}\phi_{k-1} 
}
a $U^p$-atom. Furthermore, we define the atomic space 
\EQQS{
 U^p:=\biggl{\{} u=\sum_{j=1}^{\I}\la_j a_j \, \Bigl| \, a_j : U^p \text{-atom} , \la_j \in \C \ such\ that\ \sum_{j=1}^{\I}|
          \la_j|< \I \biggr{\}}
}
with norm 
\EQQS{
 \| u\|_{U^p}:=\inf \biggl{\{} \sum_{j=1}^{\I}|\la_j| \, \Bigl| \, u=\sum_{j=1}^{\I}\la_j a_j, \la_j\in \C, a_j : U^p \text{-atom}
                   \biggr{\}}.
} 
\end{defn} 
\begin{prop}
Let $1\le p<q<\I.$ \\
(i) $U^p$ is a Banach space. \\
(ii) The embeddings $U^p\subset U^q\subset L^{\I}_t(\R;L^2_x)$ are continuous. \\
(iii) For $u\in U^p$, it holds that $\lim_{t\to t_{0}+}\|u(t)-u(t_{0})\|_{L^2_x}=0,$ i.e. every $u\in U^p$ is right-continuous. \\
(iv) The closed subspace $U^p_c$ of all continuous functions in $U^p$ is a Banach space.
\end{prop}
The above proposition is in ~\cite{HHK} (Proposition 2.2).  
\begin{defn}
Let $1\le p<\I$. We define $V^p$ as the normed space of all functions $v:\R\to L^2_x$ such that 
$\lim_{t\to \pm \I}v(t)$ exist and for which the norm 
\EQQS{
\| v\|_{V^p}:=\sup_{\{ t_k\}_{k=0}^K\in \mathcal{Z}}\Bigl(\sum_{k=1}^K\| v(t_k)-v(t_{k-1})\|_{L^2_x}^p\Bigr)^{1/p}  
} 
is finite, where we use the convention that $v(-\I):=\lim_{t\to -\I}v(t)$ and $v(\I):=0.$ 
Likewise, let $V_-^p$ denote the closed subspace of all $v \in V^p$ with $\lim_{t\to -\I}v(t)=0.$ 
\end{defn}
The definitions of $V^p$ and $V^p_-$, see also ~\cite{HHK2}.
\begin{prop} \label{embedding}
Let $1\le p<q<\I$. \\
(i) Let $v:\R \to L^2_x$ be such that 
\EQQS{
 \| v\|_{V^p_0}:= \sup_{ \{ t_k\}_{k=0}^K\in \mathcal{Z}_0}\Bigl( \sum_{k=1}^K\| v(t_{k})-v(t_{k-1})\|_{L^2_x}^p\Bigr)^{1/p}
}
is finite. Then, it follows that $v(t_0^+):=\lim_{t\to t_0+} v(t)$ exists for all $t_0\in [-\I,\I)$ and 
$v(t_0^{-}):=\lim_{t\to t_0-} v(t)$ exists for all $t_0\in (-\I,\I]$ and moreover, 
\EQQS{
 \| v\|_{V^p}=\| v\|_{V^p_0}.
}
(ii) We define the closed subspace $V^p_{rc}\, (V^p_{-,\, rc})$ of all right-continuous $V^p$ functions ($V^p_-$ functions). 
The spaces $V^p,\ V^p_{rc},\ V^p_-$ and $V^p_{-,\, rc}$ are Banach spaces. \\
(iii) The embeddings $U^p\subset V^p_{-,\, rc}\subset U^q$ are continuous. \\
(iv) The embeddings $V^p\subset V^q$ and $V^p_-\subset V^q_-$ are continuous. 
\end{prop}
The proof of Proposition \ref{embedding} is in ~\cite{HHK} (Proposition 2.4 and Corollary 2.6). 
Let $\{ \F_{\xi}^{-1}[\varphi_n](x)\}_{n\in \Z}\subset \mathcal{S}(\R^d)$ be the Littlewood-Paley decomposition with respect 
to $x$, that is to say
\EQQS{
\begin{cases}
 \varphi(\xi) \ge 0, \\
 \supp \varphi(\xi) = \{ \xi \,|\, 2^{-1} \le |\xi| \le 2\}, 
\end{cases}
}
\EQQS{
 \varphi_{n}(\xi) := \varphi(2^{-n}\xi),\ \sum_{n=-\I}^{\I}\varphi_{n}(\xi)=1\ (\, \xi \neq 0),\ 
 \psi(\xi) := 1-\sum_{n=0}^{\I}\varphi_{n}(\xi).
} 
Let $N=2^n\ (n \in \Z)$ be dyadic number. $P_N$ and $P_{<1}$ denote 
\EQQS{
 &\F_{x}[P_N f](\xi):=\varphi(\xi/N)\F_x[f](\xi)=\varphi_n(\xi)\F_x[f](\xi), \\
 &\F_{x}[P_{<1} f](\xi):=\psi(\xi)\F_x[f](\xi).
} 
Similarly, let $\til{Q}_N$ be 
\EQQS{
 \F_t [\til{Q}_N g](\ta):=\phi(\ta/N)\F_t[g](\ta)=\phi_n(\ta)\F_t[g](\ta),
}
where  $\{ \F_{\ta}^{-1}[\phi_{n}](t)\}_{n\in \Z}\subset \mathcal{S}(\R)$ be the Littlewood-Paley decomposition with respect to $t$. Let $K_{\pm}(t)=\exp \{\mp it (1-\laplacian)^{1/2} \}: L^2_x\to L^2_x$ be the Klein-Gordon unitary operator such that $\F_x[K_{\pm}(t)u_0](\xi ) = \exp \{\mp it \LR{\xi} \}\, \F_x[u_0](\xi ).$ Similarly, we define the wave unitary operator $W_{\pm c}(t)=\exp \{\mp ict (-\laplacian)^{1/2}\}:L^2_x\to L^2_x$ such that $\F_x[W_{\pm c}(t)n_0](\xi ) = \exp \{\mp ict |\xi |\}\, \F_x[n_0](\xi ).$ We set 
\EQQS{
 &W_L^{\pm c} := \bigl{\{} (\ta, \xi) \in \R \cross \R^d \, |\, L/2 \le \bigl| \ta \pm c|\xi| \bigr| \le 2L\bigr{\}}, \\ 
 &KG_L^{\pm} := \bigl{\{} (\ta, \xi) \in \R \cross \R^d \, |\, L/2 \le \bigl| \ta \pm \LR{\xi}\bigr| \le 2L\bigr{\}}. 
}
\begin{defn}
We define \\ 
 $(i)\, U^p_{K_{\pm}}=K_{\pm}(\cdot)U^p$ with norm $\| u\|_{U^p_{K_{\pm}}}=\|K_{\pm}(-\cdot )u\|_{U^p},$\\
 $(ii)\, V^p_{K_{\pm}}=K_{\pm}(\cdot)V^p$ with norm $\| u\|_{V^p_{K_{\pm}}}=\|K_{\pm}(-\cdot )u\|_{V^p}.$\\
For dyadic numbers $N, M$, 
\EQQS{
 Q_N^{K_{\pm}} := K_{\pm}(\cdot)\til{Q}_N K_{\pm}(-\cdot), \quad 
 Q_{\ge M}^{K_{\pm}} := \sum_{N \ge M} Q_N, \quad 
 Q_{<M}^{K_{\pm}} := Id - Q_{\ge M}^{K_{\pm}}.
}
Here summation over $N$ means summation over $n\in \Z$. 
Similarly, we define $U^p_{W_{\pm c}}, V^p_{W_{\pm c}}$. 
\end{defn}

\begin{rem} \label{embed}
For $L^2_x$ unitary operator $A=K_{\pm}$ or $W_{\pm c},$ 
\EQQS{
 U^2_A \subset V^2_{-,\, rc,\, A} \subset L^{\I}(\R; L^2_x)
}
\end{rem}
\begin{defn} \label{defX} 
For the Klein-Gordon equation, we define $Y^s_{K_{\pm}}$ (resp. $Z^s_{K_{\pm}})$ as the closure of all $u \in C(\R; H^s_x(\R^d)) \cap \LR{\na_x}^{-s} V^2_{-,\, rc,\, K_{\pm}}$ (resp. $u \in C(\R; H^s_x(\R^d)) \cap \LR{\na_x}^{-s} U^2_{K_{\pm}}$) with $Y^s_{K_{\pm}}$ (resp. $Z^s_{K_{\pm}}$) norm, where    
\EQQS{
  &\|u\|_{Y^s_{K_{\pm}}} := \|P_{<1}u\|_{V^2_{K_{\pm}}} + \Bigl(\sum_{N \ge 1} N^{2s}\|P_N u\|^2_{V^2_{K_{\pm}}}\Bigr)^{1/2}, \\
  &\|u\|_{Z^s_{K_{\pm}}} := \|P_{<1}u\|_{U^2_{K_{\pm}}} + \Bigl(\sum_{N \ge 1} N^{2s}\|P_N u\|^2_{U^2_{K_{\pm}}}\Bigr)^{1/2}. 
} 
For the wave equation, we define $\dot{Y}^s_{W_{\pm c}}, \dot{Z}^s_{W_{\pm c}}$ as the closure of all $n \in C(\R; H^s_x(\R^d)) \cap |\na_x|^{-s} V^2_{-,\, rc,\, W_{\pm c}}$ (resp. $n \in C(\R; H^s_x(\R^d)) \cap |\na_x|^{-s} U^2_{W_{\pm c}}$) with $\dot{Y}^s_{W_{\pm c}}$ (resp. $\dot{Z}^s_{W_{\pm c}}$) norm, where    
\EQQS{
  \|n\|_{\dot{Y}^s_{W_{\pm c}}} := \Bigl(\sum_N N^{2s}\|P_N n\|^2_{V^2_{W_{\pm c}}}\Bigr)^{1/2}, \qquad 
  \|n\|_{\dot{Z}_{W_{\pm c}}} := \Bigl(\sum_N N^{2s}\|P_N n\|^2_{U^2_{W_{\pm c}}}\Bigr)^{1/2}. 
} 
\end{defn}
\begin{defn}
For a Hilbert space $H$ and a Banach space $X\subset C(\R; H)$, we define
\EQQS{
 &B_r(H):=\{ f \in H\, |\, \|f\|_H \le r \}, \\
 &X([0,T)):=\{ u \in C([0,T);H)\, |\, \exists \til{u} \in X , \til{u}(t)=u(t), t \in [0,T) \} 
}
endowed with the norm $\|u\|_{X([0,T))}=\inf \{ \|\til{u}\|_X |\,  \til{u}(t)=u(t), t \in [0,T)\}.$
\end{defn}
We denote the Duhamel term 
\EQQS{
 &I_{T, K_{\pm}}(n,v) := \pm \int_0^t \1_{[0,T]}(t')K_{\pm}(t-t')n(t')(\om_1^{-1}v(t'))dt', \\
 &I_{T, W_{\pm c}}(u,v) := \pm \int_0^t \1_{[0,T]}(t')W_{\pm c}(t-t')\om \bigl( (\om_1^{-1}u(t')) (\ol{\om_1^{-1}v(t')})\bigr) dt'
}
for the Klein-Gordon equation and the wave equation respectively. 
The following proposition is in ~\cite{HHK} (Theorem 2.8 and Proposition 2.10). 
\begin{prop}  \label{U2norm}
Let $u\in V^1_{-, \, rc} \sub U^2$ be absolutely continuous on compact intervals. 
Then, $\|u\|_{U^2}=\ds \sup_{v \in V^2,\, \|v\|_{V^2}=1}\Bigl|\int_{-\I}^{\I}\LR{u'(t), v(t)}_{L^2_x} dt\Bigr|.$
\end{prop}
\begin{cor}  \label{U2A}
Let $A=K_{\pm}$ or $W_{\pm c}$ and $u \in V^1_{-,\, rc,\, A} \subset U^2_A$ be absolutely continuous on compact intervals.
Then, 
\EQQS{
 \|u\|_{U^2_A} = \sup_{v \in V^2_A,\, \|v\|_{V^2_A}=1}
                      \Bigl|\int_{-\I}^{\I} \LR{A(t)(A(-\cdot )u)'(t), v(t)}_{L^2_x} dt\Bigr|.
}
\end{cor}
\begin{prop} \label{mlinear}
Let $T_0:\, L^2_x \cross \dots \cross L^2_x\to L^1_{loc}(\R^d;\C)\ $ 
be a n-linear operator.
Assume that for some $1 \le p < \I$ and $1\le q \le \I$, it holds that
$$\|T_0( K_{\pm}(\cdot)\phi_1, \dots ,K_{\pm}(\cdot)\phi_n)\|_{L^p_t(\R;L^q_x(\R^d))}
   \lec \ds \prod_{i=1}^n\|\phi_i\|_{L^2_x}.$$
Then, there exists $T:U_{K_{\pm}}^p \cross \dots \cross U_{K_{\pm}}^p\to 
      L^p_t(\R;L^q_x(\R^d))$ satisfying  
$$\|T(u_1, \dots , u_n)\|_{L^p_t(\R;L^q_x(\R^d))} \lec \ds \prod_{i=1}^n
                    \|u_i\|_{U^p_{K_{\pm}}},$$
such that
$T(u_1, \dots, u_n)(t)(x)=T_0(u_1(t), \dots , u_n(t))(x)$
a.e. 
\end{prop}
See Proposition 2.19 in ~\cite{HHK} for the proof of the above proposition. 
\begin{prop} \label{Strich-w}
Let $d \ge 3, 2 \le r < \I, 2/q = (d-1)(1/2-1/r), (q, r) \neq (2, 2(d-1)/(d-3))$ and $s = 1/q-1/r+1/2$. Then it holds that 
\EQQS{
 \| W_{\pm c}(t) f\|_{L^q_t \dot{W}^{-s,r}_x(\R^{1+d})} \lec \|f\|_{L^2_x(\R^d)}. 
}
\end{prop}
For the proof of Proposition \ref{Strich-w}, see ~\cite{K}, ~\cite{GV}. 
\begin{prop} \label{Strich-kg}
Let $d \ge 3, 2 \le r < \I, 2/q = (d-1)(1/2-1/r), (q, r) \neq (2, 2(d-1)/(d-3))$ and $s = 1/q-1/r+1/2$. Then, it holds that 
\EQQS{
 \| K_{\pm}(t) f\|_{L^q_t W^{-s,r}_x(\R^{1+d})} \lec \|f\|_{L^2_x(\R^d)}. 
}
\end{prop}
For the proof of Proposition \ref{Strich-kg}, see ~\cite{MNO2}. 
Combining Proposition \ref{embedding}, Proposition \ref{mlinear}, Proposition \ref{Strich-w} and Proposition \ref{Strich-kg}, we have the following.   
\begin{prop} \label{Str}
Let $d \ge 3, 2 \le r < \I, 2/q = (d-1)(1/2-1/r), (q, r) \neq (2, 2(d-1)/(d-3))$ and $s = 1/q-1/r+1/2$.  
If $p < q$, then it holds that 
\EQQS{
 \|f\|_{L^q_t W^{-s,r}_x(\R^{1+d})} \lec \|f\|_{V^p_{K_{\pm}}}, \qquad 
 \|f\|_{L^q_t \dot{W}^{-s,r}_x(\R^{1+d})} \lec \|f\|_{V^p_{W_{\pm c}}}.
}
\end{prop}
\begin{prop} \label{unique}  
(i) Let $T>0$ and $u \in Y^s_{K_{\pm}}([0,T]), u(0)=0$. Then, there exists $0 \le T' \le T$ such that 
$\|u\|_{Y^s_{K_{\pm}}([0,T'])} < \e $.   \\
(ii) Let $T>0$ and $n \in \dot{Y}^s_{W_{\pm c}}([0,T]), n(0)=0$. 
Then, there exists $0 \le T' \le T$ such that $\|n\|_{\dot{Y}^s_{W_{\pm c}}([0,T'])} < \e$. 
\end{prop}
For the proofs of $(i)$ and $(ii)$, see Proposition 2.24 in \cite{HHK}.

\begin{lem} \label{estX}
Let $a \ge 0$. Then for $A=K_{\pm}$ or $W_{\pm c}$, it holds that 
\EQQS{
 \|\LR{\na _x}^a f\|_{V^2_A} \lec \|f\|_{Y^a_A}. 
}
\end{lem}
\begin{proof}
We only prove for $A=K_{\pm}$ since we can prove similarly for $A=W_{\pm c}$. 
By $L^2_x$ orthogonality, we have  
\EQQS{ 
  \|\LR{\na _x}^a f\|_{V^2_{K_{\pm}}}^2 
    &\lec \sup_{\{t_i\}_{i=0}^I \in \mathcal{Z}} 
              \sum_{i = 1}^I (\|P_{<1}(K_{\pm}(-t_i)f(t_i)-K_{\pm}(-t_{i-1})f(t_{i-1}))\|_{L^2_x}^2 \notag \\
    &\quad               
              + \sum_{N \ge 1} N^{2a} \|P_N(K_{\pm}(-t_i)f(t_i)-K_{\pm}(-t_{i-1})f(t_{i-1}))\|_{L^2_x}^2) \notag \\
    &\lec \sup_{\{t_i\}_{i=0}^I \in \mathcal{Z}} 
              \sum_{i=1}^I \|K_{\pm}(-t_i)P_{<1}f(t_i)-K_{\pm}(-t_{i-1})P_{<1}f(t_{i-1})\|_{L^2_x}^2 \notag \\
    &\quad    
              + \sum_{N \ge 1} N^{2a} \sup_{\{t_i\}_{i=0}^I \in \mathcal{Z}} 
                   \sum_{i=1}^I \|K_{\pm}(-t_i)P_N f(t_i)-K_{\pm}(-t_{i-1})P_N f(t_{i-1})\|_{L^2_x}^2 \notag \\
    &\lec \|f\|_{Y^a_{K_{\pm}}}^2.
}
\end{proof}
\begin{rem} \label{estx}
Similarly, we see 
\EQQS{
 \| |\na _x|^a f\|_{V^2_A} \lec \|f\|_{\dot{Y}^a_A}.  
}
\end{rem}

\begin{lem} \label{Mhigh}
If $f, g$ are measurable functions, then for $Q = Q_{<M}^A$ or $Q_{\ge M}^A, A = K_{\pm}$ or $W_{\pm c}$, it holds that 
\EQQS{
 \int_{\R^{1+d}} f(t,x) \ol{Q g(t,x)} dxdt = \int_{\R^{1+d}} \bigl( Q f(t,x)\bigr) \ol{g(t,x)}dxdt. 
}
\end{lem}
For the proof of Lemma \ref{Mhigh}, see ~\cite{KaT}, Lemma 2.17. Since $Q_{<M}^A = Id - Q_{\ge M}^A$, we also obtain the result for $Q = Q_{<M}^A$.

\begin{prop} \label{modulation}
It holds that 
\EQS{
 &\|Q_M^{K_{\pm}} u\|_{L^2_{t,x}(\R^{1+d})} \lec M^{-1/2}\|u\|_{V^2_{K_{\pm}}},\ \ \ \ \ 
      \|Q_{\ge M}^{K_{\pm}} u\|_{L^2_{t,x}(\R^{1+d})} \lec M^{-1/2}\|u\|_{V^2_{K_{\pm}}}, \label{mod} \\
  &\|Q_{<M}^{K_{\pm}} u\|_{V^2_{K_{\pm}}} \lec \|u\|_{V^2_{K_{\pm}}},\ \ \ \ \ 
      \|Q_{\ge M}^{K_{\pm}} u\|_{V^2_{K_{\pm}}} \lec \|u\|_{V^2_{K_{\pm}}}, \notag \\
  &\|Q_{<M}^{K_{\pm}} u\|_{U^2_{K_{\pm}}} \lec \|u\|_{U^2_{K_{\pm}}},\ \ \ \ \ 
      \|Q_{\ge M}^{K_{\pm}} u\|_{U^2_{K_{\pm}}} \lec \|u\|_{U^2_{K_{\pm}}}. \notag 
}
The same estimates hold by replacing the Klein-Gordon operator $K_{\pm}$ by the wave operator $W_{\pm c}.$  
\end{prop}

\begin{lem} \label{Recovery}
Let $c > 0, c \neq 1$ and $\ta_3=\ta_1-\ta_2,\ \xi_3=\xi_1-\xi_2.$ 
If $|\xi_1| \gg \LR{\xi_2}$ or $\LR{\xi_1} \ll |\xi_2|,$ then it holds that 
\EQS{ \label{recover}
 \max\big\{ \big|\ta_1 \pm \LR{\xi_1} \big|, \big|\ta_2 \pm \LR{\xi_2} \big|, \big|\ta_3 \pm c|\xi_3| \big|  \big\} 
  \gec \max\{|\xi_1|, |\xi_2|\}.
} 
\end{lem}
\begin{proof}
We only prove the case $|\xi_1| \gg \LR{\xi_2}$ since the case $\LR{\xi_1} \ll |\xi_2|$ is proved by the same manner.  
\EQS{
 (\text{l.h.s.}) \gec \bigl| \bigl(\ta_1 \pm \bigl(1+|\xi_1|)\bigr) - \bigl(\ta_2 \pm (1+|\xi_2|)\bigr) - (\ta_3 \pm c|\xi_3|)\bigr|     \label{recov}
}
If $0 < c < 1$, then we take $\e_c$ such that $0 < \e_c < (1-c)/(1+c), |\xi_2| \le \e_c |\xi_1|$.  
Then, the right hand side of \eqref{recov} is bounded by        
\EQQS{
 (1+|\xi_1|) - (1+|\xi_2|) - c|\xi_1 - \xi_2| 
   \ge |\xi_1| - \e_c |\xi_1| - c(1+\e_c)|\xi_1|   
   \gec |\xi_1|.
}
If $c > 1$, then we take $\til{\e}_c$ such that $0 < \til{\e}_c < (c-1)/(c+3), |\xi_2| \le \til{\e}_c |\xi_1|, 
|\xi_1| \ge 1/\til{\e}_c$. 
Then, the right hand side of \eqref{recov} is bounded by 
\EQQS{
 c|\xi_1 - \xi_2| - (1+|\xi_1|) - (1+|\xi_2|) 
   \ge c(1-\til{\e}_c)|\xi_1| - (1+\til{\e}_c)|\xi_1| - 2\til{\e}_c|\xi_1|  
   \gec |\xi_1|.
}   
\end{proof}
\begin{rem} \label{remconst1}
From \eqref{mod} and \eqref{recover}, we can obtain a half derivative.   
\end{rem}

\begin{lem}  \label{tri}
Let $\til{u}_{N_1}:=\1_{[0,T)}P_{N_1} u, \til{v}_{N_2}:=\1_{[0,T)}P_{N_2} v, \til{n}_{N_3}:=\1_{[0,T)}P_{N_3} n, 
Q_1, Q_2 \in \{ Q_{<M}^{K_{\pm}}, Q_{\ge M}^{K_{\pm}}\}, Q_3 \in \{ Q_{<M}^{W_{\pm c}}, Q_{\ge M}^{W_{\pm c}} \}$. Let $s = s_c = d/2-2$. Then the following estimates hold for all $0<T<\I$ : \\ 
 (i) If $N_3 \lec N_2 \sim N_1$, then  
\EQQS{
 |I_1| := \Bigl|\int_{\R^{1+d}}(\om_1^{-1}\til{u}_{N_1})(\ol{\om_1^{-1}\til{v}_{N_2}})(\ol{\om \til{n}_{N_3}})dxdt\Bigr| 
  \lec N_3^{s}\|u_{N_1}\|_{V^2_{K_{\pm}}}\|v_{N_2}\|_{V^2_{K_{\pm}}} \|n_{N_3}\|_{V^2_{W_{\pm c}}}. 
}
 (ii) It holds that 
\EQQS{
 |I_2| := \Bigl|\int_{\R^{1+d}} \til{n} (\om_1^{-1}\til{v})(\ol{P_{<1}\til{u}})dxdt\Bigr|
                       \lec \|n\|_{\dot{Y}^s_{W_{\pm c}}}\|v\|_{Y^s_{K_{\pm}}}\|P_{<1}u\|_{V^2_{K_{\pm}}}. 
}
 (iii) If $N_1 \sim N_2$, then  
\EQQS{
  |I_3| := \Bigl|\int_{\R^{1+d}}\Bigl(\sum_{N_3 \lec N_2}\til{n}_{N_3}\Bigr)(\om_1^{-1}\til{v}_{N_2})\ol{\til{u}_{N_1}}dxdt\Bigr|
                   \lec \|n\|_{\dot{Y}^s_{W_{\pm c}}}\|v_{N_2}\|_{V^2_{K_{\pm}}}\|u_{N_1}\|_{V^2_{K_{\pm}}}.  
}
 (iv) If $N_1 \sim N_3, N_1 \gg 1, M=\e N_1$ and $\e > 0$ is sufficiently small, then   
\EQQS{
   |I_i| \lec \|n_{N_3}\|_{V^2_{W_{\pm c}}} \|v\|_{Y^s_{K_{\pm}}} \|u_{N_1}\|_{V^2_{K_{\pm}}}, \quad (i=4,5)
} 
where 
\EQQS{ 
 &I_4 := \int_{\R^{1+d}}(Q_{\ge M}^{W_{\pm c}}\til{n}_{N_3})\Bigl(\sum_{N_2 \ll N_1}Q_2 \om_1^{-1}\til{v}_{N_2}\Bigr)
                      (\ol{Q_1 \til{u}_{N_1}})dxdt, \\
 &I_5 := \int_{\R^{1+d}}(Q_3 \til{n}_{N_3})\Bigl(\sum_{N_2 \ll N_1}Q_2 \om_1^{-1}\til{v}_{N_2}\Bigr)
                      (\ol{Q_{\ge M}^{K_{\pm}} \til{u}_{N_1}})dxdt.
}
\end{lem}
\begin{proof}
We show $(i)$ first.  
For $f \in V^2_A,\ A \in \{ K_{\pm}, W_{\pm c} \}$, we see 
\EQS{  \label{c}
 \|\1_{[0,T)}f\|_{V^2_A} \lec \|f\|_{V^2_A}.
}
For $d \ge 5$, we apply the H\"{o}lder inequality to have  
\EQS{
 |I_1| 
  \lec \|\om_1^{-1}\til{u}_{N_1}\|_{L^{2(d+1)/(d-1)}_{t,x}} \|\om_1^{-1}\til{v}_{N_2}\|_{L^{2(d+1)/(d-1)}_{t,x}}
                         \|\om \til{n}_{N_3}\|_{L^{(d+1)/2}_{t,x}}.                                             \label{5dn} 
}
We apply Proposition \ref{Str}, \eqref{c} and the Sobolev inequality, then we have 
\EQS{
  &\|\om_1^{-1}\til{f}_N\|_{L^{2(d+1)/(d-1)}_{t,x}}
      \lec \LR{N}^{1/2-1} \|f_N\|_{V^2_{K_{\pm}}} 
        =     \LR{N}^{-1/2} \|f_N\|_{V^2_{K_{\pm}}},                                                        \label{5dn2} \\
  &\|\om \til{n}_{N_3}\|_{L^{(d+1)/2}_{t,x}} 
      \lec \| |\na_x|^{d(d-5)/2(d-1)}\om \til{n}_{N_3}\|_{L^{(d+1)/2}_t L^{2(d^2-1)/(d^2-9)}_x}  \notag \\
      &\qquad \qquad \qquad \lec \| |\na_x|^{d/2-2}\om \til{n}_{N_3}\|_{V^2_{W_{\pm c}}}                                     \label{5dn3} \\
      &\qquad \qquad \qquad \lec N_3^{s_c + 1}\|n_{N_3}\|_{V^2_{W_{\pm c}}}                                                    \label{5dn4} 
}
Collecting \eqref{5dn}, \eqref{5dn2}, \eqref{5dn4} and $N_3 \lec N_1 \sim N_2$, we obtain 
\EQQS{
 |I_1| \lec N_3^{s_c} \|u_{N_1}\|_{V^2_{K_{\pm}}} \|v_{N_2}\|_{V^2_{K_{\pm}}} \|n_{N_3}\|_{V^2_{W_{\pm c}}}.     
}
Next, we prove $(ii)$. 
For $d \ge 5$, by the H\"{o}lder inequality to have 
\EQS{
 |I_2| \lec \|\til{n}\|_{L^{(d+1)/2}_{t,x}} \|\om_1^{-1}\til{v}\|_{L^{2(d+1)/(d-1)}_{t,x}}
                 \| P_{<1}\til{u}\|_{L^{2(d+1)/(d-1)}_{t,x}}.                                                                      \label{5di2a}
}
From Proposition \ref{Str}, \eqref{5dn3}, Remark \ref{estx} and Lemma \ref{estX}, we obtain 
\EQS{ 
 &\|\til{n}\|_{L^{(d+1)/2}_{t,x}} \lec \|n\|_{\dot{Y}^{s_c}_{W_{\pm c}}},                                                     \label{5di2b} \\
 &\|\om_1^{-1}\til{v}\|_{L^{2(d+1)/(d-1)}_{t,x}} \lec \| \LR{\na_x}^{-1/2} v\|_{V^2_{K_{\pm}}} 
                                                         \lec \|\LR{\na_x}^{s_c}v\|_{V^2_{K_{\pm}}}  
                                                         \lec \|v\|_{Y^{s_c}_{K_{\pm}}},                                           \label{5di2c} \\
 &\|P_{<1}\til{u}\|_{L^{2(d+1)/(d-1)}_{t,x}} \lec \|\LR{\na_x}^{1/2} P_{<1}u\|_{V^2_{K_{\pm}}}
                                                   \lec \|P_{<1}u\|_{V^2_{K_{\pm}}}.                                              \label{5di2d} 
}
Collecting \eqref{5di2a}--\eqref{5di2d}, we obtain 
\EQQS{
 |I_2| \lec \|n\|_{\dot{Y}^{s_c}_{W_{\pm c}}} \|v\|_{Y^{s_c}_{K_{\pm}}} \|P_{<1}u\|_{V^2_{K_{\pm}}}.
}
We prove $(iii)$ for $d \ge 5$. 
We apply the H\"{o}lder inequality to have  
\EQS{
 |I_3|  
   \lec \Bigl{\|}\sum_{N_3 \lec N_2} \til{n}_{N_3}\Bigr{\|}_{L^{(d+1)/2}_{t,x}} \|\om_1^{-1}\til{v}_{N_2}\|_{L^{2(d+1)/(d-1)}_{t,x}}
              \| \til{u}_{N_1}\|_{L^{2(d+1)/(d-1)}_{t,x}}.                                                          \label{5dn7} 
}
Similar to \eqref{5dn3}, the Sobolev inequality and Proposition \ref{Str}, we have 
\EQS{
 \Bigl{\|}\sum_{N_3 \lec N_2} \til{n}_{N_3}\Bigr{\|}_{L^{(d+1)/2}_{t,x}} 
   \lec \Bigl{\|} |\na_x|^{s_c} \sum_{N_3 \lec N_2} \til{n}_{N_3}\Bigr{\|}_{V^2_{W_{\pm c}}}.     \label{ortho}
}
By the $L^2_x$ orthogonality, we obtain 
\EQS{
 \Bigl{\|} |\na_x|^{s_c} \sum_{N_3 \lec N_2} \til{n}_{N_3}\Bigr{\|}_{V^2_{W_{\pm c}}}^2 
 &\lec \sup_{\{t_i\}_{i=0}^I \in \mathcal{Z}} \sum_{i=1}^I \sum_N N^{2s_c} 
                \Bigl{\|} P_N \Bigl{\{} W_{\pm c}(-t_i)\Bigl(\sum_{N_3 \lec N_2}\til{n}_{N_3}(t_i)\Bigr) \notag \\ 
 &\qquad \qquad  - W_{\pm c}(-t_{i-1})\Bigl(\sum_{N_3 \lec N_2} \til{n}_{N_3}(t_{i-1})\Bigr)\Bigr{\}} \Bigr{\|}_{L^2_x}^2.     \label{ortho1}
} 
Since $P_N \til{n}_{N_3} = 0$ if $N_3 > 2N$ or $N_3 < N/2$ and $P_N$ is projection, the right-hand side is bounded by 
\EQS{
 &\sup_{\{t_i\}_{i=0}^I \in \mathcal{Z}} \sum_{i=1}^I \sum_N N^{2s_c} 
                \|W_{\pm c}(-t_i) P_N \til{n}(t_i) - W_{\pm c}(-t_{i-1}) P_N \til{n}(t_{i-1}) \|_{L^2_x}^2 \notag \\
  &\lec \sum_N N^{2s_c} \sup_{\{t_i\}_{i=0}^I \in \mathcal{Z}} 
                \|W_{\pm c}(-t_i) P_N \til{n}(t_i) - W_{\pm c}(-t_{i-1}) P_N \til{n}(t_{i-1}) \|_{L^2_x}^2 \notag \\
  &\lec \|n\|_{\dot{Y}^{s_c}_{W_{\pm c}}}^2.                                                                          \label{ortho2}
}
Hence, from \eqref{5dn7}--\eqref{ortho2}, \eqref{5dn2} and $N_1 \sim N_2$, we have 
\EQQS{
 |I_3|  
   &\lec \|n\|_{\dot{Y}^{s_c}_{W_{\pm c}}} \LR{N_2}^{-1/2}\|v_{N_2}\|_{V^2_{K_{\pm}}} \LR{N_1}^{1/2}\|u_{N_1}\|_{V^2_{K_{\pm}}} \\
   &\lec \|n\|_{\dot{Y}^{s_c}_{W_{\pm c}}} \|v_{N_2}\|_{V^2_{K_{\pm}}} \|u_{N_1}\|_{V^2_{K_{\pm}}}.    
}
We prove $(iv)$. 
The estimate for $I_5$ is obtained by the same manner as the estimate for $I_4$, so we only estimate $I_4$. 
We apply the H\"{o}lder inequality to have  
\EQS{
 |I_4|   
  \lec \|Q_{\ge M}^{W_{\pm c}}\til{n}_{N_3}\|_{L^2_{t,x}}\Bigl{\|}\sum_{N_2 \ll N_1}Q_2 \om_1^{-1}\til{v}_{N_2}\Bigr{\|}_{L^{d+1}_{t,x}} 
             \|Q_1 \til{u}_{N_1}\|_{L^{2(d+1)/(d-1)}_{t,x}}.                                                               \label{I4}
}
By Proposition \ref{modulation}, \eqref{5dn2} and \eqref{c}, we have 
\EQS{
 &\|Q_{\ge M}^{W_{\pm c}}\til{n}_{N_3}\|_{L^2_{t,x}} \lec N_1^{-1/2}\|n_{N_3}\|_{V^2_{W_{\pm c}}},                           \label{I4a} \\
 &\|Q_1 \til{u}_{N_1}\|_{L^{2(d+1)/(d-1)}_{t,x}} \lec \LR{N_1}^{1/2} \|u_{N_1}\|_{V^2_{K_{\pm}}}.              \label{I4b} 
}
We apply the Sobolev inequality, Proposition \ref{Str}, Proposition \ref{modulation} and \eqref{c}, we have
\EQS{
 \Bigl{\|}\sum_{N_2 \ll N_1}Q_2 \om_1^{-1}\til{v}_{N_2}\Bigr{\|}_{L^{d+1}_{t,x}} 
  &\lec \Bigl{\|} \LR{\na_x}^{d(d-3)/2(d-1)}\sum_{N_2 \ll N_1}Q_2 \om_1^{-1}
               \til{v}_{N_2}\Bigr{\|}_{L^{d+1}_t L^{2(d^2-1)/(d^2-5)}_x}  \notag \\
  &\lec \Bigl{\|} \LR{\na_x}^{d(d-3)/2(d-1)+1/(d-1)-1}\sum_{N_2 \ll N_1}\til{v}_{N_2}\Bigr{\|}_{V^2_{K_{\pm}}}. \label{I4C1}
}
Similar to \eqref{ortho1} and \eqref{ortho2}, we have 
\EQS{
 \Bigl{\|} \LR{\na_x}^{d(d-3)/2(d-1)+1/(d-1)-1}\sum_{N_2 \ll N_1}\til{v}_{N_2}\Bigr{\|}_{V^2_{K_{\pm}}}  
   \lec \|v\|_{Y^{s_c}_{K_{\pm}}}.                                                                                        \label{I4c}
}
Collecting \eqref{I4}--\eqref{I4c} and $N_1 \gg 1$, we obtain  
\EQQS{
 |I_4| \lec \|n_{N_3}\|_{V^2_{W_{\pm c}}} \|v\|_{Y^{s_c}_{K_{\pm}}} \|u_{N_1}\|_{V^2_{K_{\pm}}}. 
} 
\end{proof}

The following proposition is in ~\cite{Sc}, Proposition 10.   
\begin{prop} ($L^4$ Strichartz estimate) 
For all dyadic numbers $H \ge 1$ and $N$, it holds that   
\EQQS{
 \| W_{\pm c}(t) P_N \phi \|_{L^4_{t,x}} \lec N^{(d-1)/4}\|P_N \phi\|_{L^2_x},\qquad  \| K_{\pm}(t) P_H \varphi \|_{L^4_{t,x}} \lec H^{(d-1)/4}\|P_H \varphi\|_{L^2_x}.
}
\end{prop}

From Proposition \ref{mlinear} and the above proposition, we obtain the following.

\begin{prop}   \label{u4str}
For dyadic numbers $H \ge 1$ and $N$, it holds that 
\EQQS{
 \|u_N\|_{L^4_{t,x}} \lec N^{(d-1)/4} \|u_N\|_{U^4_{W_{\pm c}}},\qquad \|v_H\|_{L^4_{t,x}} \lec H^{(d-1)/4} \|v_H\|_{U^4_{K_{\pm}}}.   
}
\end{prop}


\begin{prop}   \label{x01/2}
Let $u_M, v_N \in L^2(\R^{1+d})$ be such that 
\EQQS{
 \supp \F u_M \subset W_{L_1}^{\pm c} \cap \bigl(\R \cross (C \cap P_M)\bigr),\qquad \supp \F v_N \subset KG_{L_2}^{\pm} \cap P_{N} 
}
for dyadic numbers $L_1, L_2, M, N$ and a cube $C \subset \R^d$ of side length $e$. 
If $L \ll M \sim N, c>0$ and $c \neq 1$, it holds that  
\EQQS{
 \|P_L(u_M v_N)\|_{L^2_{t,x}} \lec L^{(d-1)/2} (L_1 L_2)^{1/2} \|u_M\|_{L^2_{t,x}} \|v_N\|_{L^2_{t,x}}. 
}
\end{prop}
\begin{proof}
Let $f := \F u_M, g := \F v_N$. 
By the Cauchy-Schwarz inequality, we have 
\EQQS{
 \Bigl{\|} \int_{|\xi| \sim L} f(\ta_1, \xi_1) g(\ta - \ta_1, \xi - \xi_1) d\ta_1 d\xi_1 \Bigr{\|}_{L^2_{\ta, \xi}} \lec \sup_{\ta,\,  \xi} |E(\ta, \xi)|^{1/2} \|f\|_{L^2} \|g\|_{L^2}
}
where 
\EQQS{
 E(\ta, \xi) = \{ (\ta_1, \xi_1) \in \supp f ; (\ta - \ta_1, \xi - \xi_1) \in \supp g, |\xi| \sim L \} \subset \R^{1+d}.
}
Put $\underline{l} := \min\{ L_1, L_2 \}, \ol{l} := \max\{ L_1, L_2 \}$. 
By the Fubini theorem, 
\EQQS{
 |E(\ta, \xi)| \le \underline{l} \left|  \{ \xi_1 ; \left| \ta \pm c|\xi_1| \pm |\xi - \xi_1| \right| \lec \ol{l}, \xi_1 \in C, 
                                  |\xi_1| \sim M, |\xi - \xi_1| \sim N, |\xi| \sim L\} \right|.
}
In the right-hand side of the above inequality, the subset of the $\xi_1$ is contained in a cube of side length $m$, where $m \sim \min\{e, N\} \sim L$. 
For some $i \in \{1,...,d\}$, we set $|(\xi - \xi_1)_i| \gec N$, where $(\xi - \xi_1)_i$ denotes the $i$-th component of $\xi - \xi_1$. 
We compute  
\EQS{
 |\p_{\xi_{1,i}}(\ta \pm c|\xi_1| \pm (1+|\xi-\xi_1|))| 
 = \left|\frac{(\xi - \xi_1)_i}{|\xi - \xi_1|} \pm c \frac{\xi_{1,i}}{|\xi_1|}\right|,     \label{Mod}
}
where $\xi_{1,i}$ be the $i$-th component of $\xi_1$. 
Since $|(\xi - \xi_1)_i| \gec N$ and $|\xi| \sim L$, it suffices to consider the case $|\xi_{0,i}| \ll |\xi_{1,i}|$, where $\xi_{0,i}$ be the $i$-th component of $\xi$.   
Firstly, we consider the case $0 < c \ll 1$. We have 
\EQQS{
 r.h.s.\ of\ \eqref{Mod} \ge \frac{|(\xi-\xi_1)_i|}{|\xi - \xi_1|} - c\frac{|\xi_{1,i}|}{|\xi_1|} 
                                   \gec 1-c              
}   
from $|(\xi - \xi_1)_i| \gec N \sim |\xi -\xi_1|$ and $|\xi_1| \ge |\xi_{1,i}|$. 
Secondly, we consider the case $c \sim 1, c \neq 1$. The assumption $L \ll N \sim M$ implies $(1-\e)|\xi-\xi_1| \le |\xi_1| \le (1+\e)|\xi-\xi_1|$ for sufficiently small $\e > 0$. From the above inequality and  $|\xi_{0,i}| \ll |\xi_{1,i}|$, we obtain  
\EQQS{
 r.h.s.\ of\ \eqref{Mod} \gec \left| c\frac{|\xi_{1,i}|}{|\xi_1|} - \frac{|(\xi-\xi_1)_i|}{|\xi - \xi_1|} \right| 
                                   \gec |c - 1|.
}
Finally, we consider the case $c \gg 1$. We have   
\EQQS{
 r.h.s.\ of\ \eqref{Mod} \gec c\frac{|\xi_{1,i}|}{|\xi_1|} - \frac{|(\xi - \xi_1)_i|}{|\xi - \xi_1|} 
                                   \gec c-1                
}   
since $|(\xi - \xi_1)_i| \gec N$ and $|\xi_{0,i}| \ll |\xi_{1,i}|$.  
Therefore,  
\EQS{
  |\p_{\xi_{1,i}}(\ta \pm c|\xi_1| \pm (1+|\xi-\xi_1|))| \gec |c - 1|.   \label{deriv}
}
Hence by \eqref{deriv} and the mean value theorem, we have 
\EQQS{
 &\left|  \{ \xi_1 ; \left| \ta \pm c|\xi_1| \pm |\xi - \xi_1| \right| \lec \ol{l}, \xi_1 \in C, |\xi_1| \sim M, |\xi - \xi_1| \sim N, |\xi| \sim L\} \right| \\
 &\lec |c-1|^{-1}m^{d-1} \ol{l}. 
}
From $m \sim L$, we have 
\EQQS{
 |E(\xi, \ta)|^{1/2} \lec (\underline{l} |c-1|^{-1} m^{d-1}\ol{l})^{1/2} 
                           \sim |c-1|^{-1/2}(L_1 L_2)^{1/2}L^{(d-1)/2}. 
}
Thus, we obtain the result.  
\end{proof}

Proposition \ref{x01/2} implies the following.   
\begin{prop} \label{L2lin}
Let $L \ll M \sim N, c>0$ and $c \neq 1$. For $u_M = W_{\pm c}(t) P_M \phi, v_N = K_{\pm}(t) P_N \varphi$, it holds that  
\EQQS{
 \|P_L(u_M v_N)\|_{L^2_{t,x}} \lec L^{(d-1)/2} \|P_M \phi\|_{L^2_x} \|P_N \varphi\|_{L^2_x}.
}
\end{prop}
 
From Proposition \ref{mlinear} and the above proposition, we have the following.
\begin{prop}   \label{u2-be1}
Let $L \ll M \sim N, c>0$ and $c \neq 1$. It holds that 
\EQQS{
 \|P_L(u_M v_N)\|_{L^2_{t,x}} \lec L^{(d-1)/2} \|u_M\|_{U^2_{W_{\pm c}}} \|v_N\|_{U^2_{K_{\pm}}}.    
} 
\end{prop}
The following proposion is in ~\cite{HHK}, Proposition 2.20.
\begin{prop}   \label{hokan}
Let $q > 1, E$ be a Banach space, $A = K_{\pm}$ or $W_{\pm c}$ and $T:U^q_A \to E$ be a bounded, linear operator with $\|Tu\|_E \le C_q\|u\|_{U^q_A}$ for all $u \in U^q_A$. In addition, assume that for some $1 \le p < q$ there exists $C_p \in (0, C_q]$ such that the estimate $\|Tu\|_E \le C_p\|u\|_{U^p_A}$ holds true for all $u \in U^p_A$. Then, $T$ satisfies the estimate 
\EQQS{
 \|Tu\|_E \le C_p\bigl( 1 + \ln (C_q/C_p) \bigr)\|u\|_{V^p_A}, \quad u \in V^p_A. 
} 
\end{prop}
\begin{prop} \label{bi-str}
Let $L \ll M \sim N, N \ge 1, c>0$ and $c \neq 1$. For sufficiently small $\e > 0$, it holds that 
\EQQS{
 \|P_L(u_M v_N)\|_{L^2_{t,x}} \lec L^{(d-1)/2} \left( M/L \right)^{\e} \|u_M\|_{V^2_{W_{\pm c}}} \|v_N\|_{V^2_{K_{\pm}}}. 
}
\end{prop}

\begin{proof} 
By the H\"{o}lder inequality, $M \sim N, N \ge 1$ and Proposition \ref{u4str}, we obtain  
\EQS{
 \|P_L(u_M v_N)\|_{L^2_{t,x}} \lec \|u_M\|_{L^4_{t,x}} \|v_N\|_{L^4_{t,x}} 
                                               \lec M^{(d-1)/2}\|u_M\|_{U^4_{W_{\pm c}}} \|v_N\|_{U^4_{K_{\pm}}}.       \label{L4}                                     
} 
Let $Sv := P_L(\til{P}_M u \til{P}_N v)$, where $\til{P}_M = P_{M/2} + P_M + P_{2M}$, such that $\til{P}_M P_M = P_M. \til{P}_N$ is defined by the same manner as $\til{P}_M$.    
From \eqref{L4} and $U^2_{W_{\pm c}} \subset U^4_{W_{\pm c}}$, we have  
\EQS{
 \|S\|_{U^4_{K_{\pm}} \to L^2} \lec M^{(d-1)/2}\|u\|_{U^4_{W_{\pm c}}} 
                                            \lec M^{(d-1)/2}\|u\|_{U^2_{W_{\pm c}}}.   \label{t1}
}
From Proposition \ref{u2-be1}, we have 
\EQS{
 \|S\|_{U^2_{K_{\pm}} \to L^2} \lec L^{(d-1)/2}\|u\|_{U^2_{W_{\pm c}}}.   \label{t2}
}
From \eqref{t1}, \eqref{t2} and Proposition \ref{hokan}, for sufficiently small $\e' > 0$, we have   
\EQS{
 \|S\|_{V^2_{K_{\pm}} \to L^2} \lec L^{(d-1)/2} \left( M/L \right)^{\e'} \|u\|_{U^2_{W_{\pm c}}}.      \label{v2bs1}
}
Let $Tu := P_L(\til{P}_M u \til{P}_N v)$. 
From Proposition \ref{u4str}, $M \sim N$ and $V^2_{K_{\pm}} \subset U^4_{K_{\pm}}$, we have    
\EQS{
 \|T\|_{U^4_{W_{\pm c}} \to L^2} \lec N^{(d-1)/2}\|v_N\|_{U^4_{K_{\pm}}} 
                                               \lec N^{(d-1)/2}\|v_N\|_{V^2_{K_{\pm}}}
                                                \lec N^{(d-1)/2}\|v\|_{V^2_{K_{\pm}}}.   \label{s1}
}
By \eqref{v2bs1}, we have 
\EQS{
 \|T\|_{U^2_{W_{\pm c}} \to L^2} \lec L^{(d-1)/2} \left( M/L \right)^{\e'} \|v\|_{V^2_{K_{\pm}}}.   \label{s2}
}
Collecting \eqref{s1}, \eqref{s2}, $M \sim N$ and Proposition \ref{hokan}, we obtain  
\EQQS{
 \|T\|_{V^2_{W_{\pm c}} \to L^2} \lec L^{(d-1)/2} \left( M/L \right)^{2\e'} \|v\|_{V^2_{K_{\pm}}}.      
}
Taking $\e = 2\e'$, the claim follows.
\end{proof}
\section{Bilinear estimates}
\begin{prop}  \label{BE}
Let $d \ge 5, s = s_c = d/2-2$ and $c>0, c \neq 1$. Then for all $0 < T < \I$, it holds that 
\EQS{
 &\| I_{T,K_{\pm}}(n,v)\|_{Z^s_{K_{\pm}}} 
      \lec \|n\|_{\dot{Y}^s_{W_{\pm c}}} \|v\|_{Y^s_{K_{\pm}}},                     \label{BEKG} \\
 &\| I_{T,\, W_{\pm c}}(u,v)\|_{\dot{Z}^s_{W_{\pm c}}} 
      \lec \|u\|_{Y^s_{K_{\pm}}} \|v\|_{Y^s_{K_{\pm}}}.                                 \label{BEW}
} 
\end{prop}
\begin{rem} \label{remconstbe}
In \eqref{BEKG} and \eqref{BEW}, the implicit constant does not depend on $T$. 
\end{rem}
\begin{proof}
We denote $\til{u}_{N_1}:=\1_{[0,T)}P_{N_1}u, \til{v}_{N_2}:=\1_{[0,T)}P_{N_2}v, \til{n}_{N_3}:=\1_{[0,T)}P_{N_3}n$. 
To prove \eqref{BEKG}, we need to estimate the following.  
\EQQS{
  \| I_{T,K_{\pm}}(n,v)\|_{Z^{s_c}_{K_{\pm}}}^2 \lec \sum_{i=0}^3 J_i   
}
where 
\EQQS{    
 &J_0 := \Bigl{\|} \int_0^t \1_{[0,T)}(t')K_{\pm}(t-t') P_{<1}(\til{n} (\om_1^{-1} \til{v}))(t')dt'\Bigr{\|}_{U^2_{K_{\pm}}}^2, \\
 &J_1 := \sum_{N_1 \ge 1} N_1^{2s_c} \Bigl{\|} \int_0^t \1_{[0,T)}(t') K_{\pm}(t-t') \sum_{N_2 \sim N_1} \sum_{N_3 \lec N_2} 
              P_{N_1}(\til{n}_{N_3}(\om_1^{-1}\til{v}_{N_2}))(t')dt' \Bigr{\|}_{U^2_{K_\pm}}^2,  \\
 &J_2 := \sum_{N_1 \ge 1} N_1^{2s_c} \Bigl{\|} \int_0^t \1_{[0,T)}(t') K_{\pm}(t-t') \sum_{N_2 \ll N_1} \sum_{N_3 \sim N_1} 
              P_{N_1}(\til{n}_{N_3}(\om_1^{-1}\til{v}_{N_2}))(t')dt' \Bigr{\|}_{U^2_{K_\pm}}^2,  \\  
 &J_3 := \sum_{N_1 \ge 1} N_1^{2s_c} \Bigl{\|} \int_0^t \1_{[0,T)}(t') K_{\pm}(t-t') \sum_{N_2 \gg N_1} \sum_{N_3 \sim N_2} 
              P_{N_1}(\til{n}_{N_3}(\om_1^{-1}\til{v}_{N_2}))(t')dt' \Bigr{\|}_{U^2_{K_\pm}}^2. 
}
By Corollary \ref{U2A} and Lemma \ref{tri} $(ii)$, we have 
\EQS{
 J_0^{1/2} 
     &\lec \sup_{\|u\|_{V^2_{K_{\pm}}}=1} \Bigl{|}\int_{\R^{1+d}} \til{n}(\om_1^{-1}\til{v})(\ol{P_{<1}\til{u}})dxdt\Bigr{|} \notag  \\
     &\lec \|n\|_{\dot{Y}^{s_c}_{W_{\pm c}}}\|v\|_{Y^{s_c}_{K_{\pm}}}.                                         \label{j0}
}
We apply Corollary \ref{U2A}, $N_1 \sim N_2$, Lemma \ref{tri} $(iii)$ and $\|\til{u}_{N_1}\|_{V^2_{K_{\pm}}} \lec \|u\|_{V^2_{K_{\pm}}}$, then   
\EQS{
 J_1 &\lec \sum_{N_1 \ge 1} N_1^{2s_c} \sup_{\|u\|_{V^2_{K_{\pm}}}=1} \Bigl{|}\sum_{N_2 \sim N_1} \sum_{N_3 \lec N_2}
                  \int_{\R^{1+d}} \til{n}_{N_3} (\om_1^{-1}\til{v}_{N_2}) \ol{\til{u}_{N_1}} dxdt\Bigr{|}^2               \notag \\                
       &\lec \sum_{N_2 \gec 1} N_2^{2s_c} \|n\|_{\dot{Y}^{s_c}_{W_{\pm c}}}^2 \|v_{N_2}\|_{V^2_{K_{\pm}}}^2  \notag \\
       &\lec \|n\|_{\dot{Y}^{s_c}_{W_{\pm c}}}^2 \|v\|_{Y^{s_c}_{K_{\pm}}}^2.                                 \label{j1}
} 
For the estimate of $J_2$, we take $M=\e N_1$ for sufficiently small $\e >0$. 
Then, from Lemma \ref{Recovery}, we have  
\EQQS{
   &P_{N_1} Q_{<M}^{K_{\pm}} \bigl( (Q_{<M}^{W_{\pm c}} \til{n}_{N_3})(Q_{<M}^{K_{\pm}} \om_1^{-1}\til{v}_{N_2}) \bigr) \\
   &= P_{N_1}  Q_{<M}^{K_{\pm}} \Bigl[ \F^{-1}\Bigl( \int_{\ta_1 = \ta_2 + \ta_3,\, \xi_1 = \xi_2 + \xi_3} 
       \widehat{(Q_{<M}^{W_{\pm c}} \til{n}_{N_3})}(\ta_3, \xi_3) \widehat{(Q_{<M}^{K_{\pm}} \om_1^{-1}\til{v}_{N_2})}(\ta_2, \xi_2) \Bigr) \Bigr]
   =0
}
when $N_1 \gg \LR{N_2}$. Therefore, 
\EQQS{
 P_{N_1}\bigl(\til{n}_{N_3} (\om_1^{-1}\til{v}_{N_2})\bigr)= \sum_{i=1}^3 P_{N_1} F_i, 
}
where   
\EQQS{
 &F_1 := Q_1\bigl( (Q_{\ge M}^{W_{\pm c}} \til{n}_{N_3}) (Q_2 \om_1^{-1}\til{v}_{N_2})\bigr), \qquad  
   F_2 := Q_1\bigl( (Q_3 \til{n}_{N_3}) (Q_{\ge M}^{K_{\pm}} \om_1^{-1}\til{v}_{N_2})\bigr), \notag \\
 &F_3 := Q_{\ge M}^{K_{\pm}}\bigl((Q_3 \til{n}_{N_3}) (Q_2 \om_1^{-1}\til{v}_{N_2})\bigr). 
}
Here, $Q_1, Q_2 \in \{ Q_{<M}^{K_{\pm}}, Q_{\ge M}^{K_{\pm}} \}$ and  $Q_3 \in \{ Q_{<M}^{W_{\pm c}}, Q_{\ge M}^{W_{\pm c}} \}$. 
For the estimate of $F_1$, we apply Corollary \ref{U2A}, Lemma \ref{Mhigh}, Lemma \ref{tri} $(iv), N_3 \sim N_1 \ge 1$ and $\|\til{u}_{N_1}\|_{V^2_{K_{\pm}}} \lec \|u\|_{V^2_{K_{\pm}}}$, then we have 
\EQS{
 &\sum_{N_1 \ge 1} N_1^{2s_c} \Bigl{\|} \int_0^t \1_{[0,T)}(t') K_{\pm}(t-t') \sum_{N_2 \ll N_1} \sum_{N_3 \sim N_1} 
            P_{N_1}F_1(t') dt' \Bigr{\|}_{U^2_{K_\pm}}^2  \notag \\
 &\lec \sum_{N_1 \ge 1} N_1^{2s_c} \sup_{\|u\|_{V^2_{K_{\pm}}}=1} \Bigl{|} \sum_{N_2 \ll N_1} \sum_{N_3 \sim N_1} 
            \int_{\R^{1+d}} (Q_{\ge M}^{W_{\pm c}}\til{n}_{N_3}) (Q_2 \om_1^{-1}\til{v}_{N_2} ) (\ol{Q_1 \til{u}_{N_1}}) dxdt\Bigr{|}^2  \notag \\
 &\lec \sum_{N_3 \gec 1} N_3^{2s_c}  \|n_{N_3}\|_{V^2_{W_{\pm c}}}^2 \|v\|_{Y^{s_c}_{K_{\pm}}}^2    \notag \\
 &\lec \|n\|_{\dot{Y}^{s_c}_{W_{\pm c}}}^2 \|v\|_{Y^{s_c}_{K_{\pm}}}^2.                                                      \label{Jj21}
} 
For the estimate of $F_2$, we apply Corollary \ref{U2A}, Lemma \ref{Mhigh} and the triangle inequality, we have 
\EQS{
 &\sum_{N_1 \ge 1} N_1^{2s_c} \Bigl{\|} \int_0^t \1_{[0,T)}(t') K_{\pm}(t-t') \sum_{N_2 \ll N_1} \sum_{N_3 \sim N_1} 
    P_{N_1}F_2(t') dt' \Bigr{\|}_{U^2_{K_\pm}}^2 \notag \\
 &\lec \sum_{N_1 \ge 1} N_1^{2s_c} \sup_{\|u\|_{V^2_{K_{\pm}}}=1} \Bigl{|} \sum_{N_2 \ll N_1} \sum_{N_3 \sim N_1} 
            \int_{\R^{1+d}} (Q_3 \til{n}_{N_3}) (Q_{\ge M}^{K_{\pm}} \om_1^{-1}\til{v}_{N_2} ) (\ol{Q_1 \til{u}_{N_1}}) dxdt\Bigr{|}^2 \notag \\
 &\lec \sum_{N_1 \ge 1} N_1^{2s_c} \sup_{\|u\|_{V^2_{K_{\pm}}}=1} \sum_{N_2 \ll N_1} \sum_{N_3 \sim N_1} 
            \Bigl{|} \int_{\R^{1+d}} (Q_3 \til{n}_{N_3}) (Q_{\ge M}^{K_{\pm}} \om_1^{-1}\til{v}_{N_2} ) (\ol{Q_1 \til{u}_{N_1}}) dxdt\Bigr{|}^2.  \label{jj} 
}
By Proposition \ref{bi-str}, $N_2 \ll N_1 \sim N_3, N_1 \ge 1$ and Proposition \ref{modulation}, we have 
\EQS{
 &\Bigl{|} \int_{\R^{1+d}} (Q_3 \til{n}_{N_3}) (Q_{\ge M}^{K_{\pm}} \om_1^{-1}\til{v}_{N_2}) (\ol{Q_1 \til{u}_{N_1}}) dxdt\Bigr{|} \notag \\
  &\lec \|Q_{\ge M}^{K_{\pm}} \om_1^{-1}\til{v}_{N_2}\|_{L^2_{t,x}} \|P_{N_2}\bigl( (Q_3 \til{n}_{N_3})(\ol{Q_1 \til{u}_{N_1}}) \bigr)\|_{L^2_{t,x}} \notag \\
  &\lec N_3^{-1/2} \LR{N_2}^{-1} \|v_{N_2}\|_{V^2_{K_{\pm}}} N_2^{(d-1)/2}(N_3/N_2)^{\e} \|n_{N_3}\|_{V^2_{\pm c}} \|u_{N_1}\|_{V^2_{K_{\pm}}} \notag \\
  &\lec N_2^{s_c}(N_2/N_3)^{1/2-\e} \|v_{N_2}\|_{V^2_{K_{\pm}}} \|n_{N_3}\|_{V^2_{\pm c}} \|u_{N_1}\|_{V^2_{K_{\pm}}}. \label{jjj}
}
By \eqref{jjj} and the Cauchy-Schwarz inequality, the right-hand side of \eqref{jj} is bounded by  
\EQS{
 &\sum_{N_3 \gec 1} N_3^{2s_c}  \|n_{N_3}\|_{V^2_{W_{\pm c}}}^2 \Bigl( \sum_{N_2 \ll N_3} \bigl( N_2/N_3 \bigr)^{1/2-\e}N_2^{s_c}\|v_{N_2}\|_{V^2_{K_{\pm}}}\Bigr)^2  \notag \\
 &\lec \|n\|_{\dot{Y}^{s_c}_{W_{\pm c}}}^2 \|v\|_{Y^{s_c}_{K_{\pm}}}^2.                                                      \label{Jj22}
}  
For the estimate for $F_3$, we apply Corollary \ref{U2A}, Lemma \ref{Mhigh}, Lemma \ref{tri} $(iv), N_3 \sim N_1 \ge 1$ and $\|\til{u}_{N_1}\|_{V^2_{K_{\pm}}} \lec \|u\|_{V^2_{K_{\pm}}}$, then we obtain  
\EQS{
 &\sum_{N_1 \ge 1} N_1^{2s_c} \Bigl{\|} \int_0^t \1_{[0,T)}(t') K_{\pm}(t-t') \sum_{N_2 \ll N_1} \sum_{N_3 \sim N_1} 
    P_{N_1}F_3(t') dt' \Bigr{\|}_{U^2_{K_\pm}}^2 \notag \\
 &\lec \sum_{N_1 \ge 1} N_1^{2s_c} \sup_{\|u\|_{V^2_{K_{\pm}}}=1} \Bigl{|} \sum_{N_2 \ll N_1} \sum_{N_3 \sim N_1} 
            \int_{\R^{1+d}} (Q_3 \til{n}_{N_3}) (Q_2 \om_1^{-1}\til{v}_{N_2} ) (\ol{Q_{\ge M}^{K_{\pm}} \til{u}_{N_1}}) dxdt\Bigr{|}^2   \notag \\
 &\lec \sum_{N_3 \gec 1} N_3^{2s_c}  \|n_{N_3}\|_{V^2_{W_{\pm c}}}^2 \|v\|_{Y^{s_c}_{K_{\pm}}}^2   \notag \\
 &\lec \|n\|_{\dot{Y}^{s_c}_{W_{\pm c}}}^2 \|v\|_{Y^{s_c}_{K_{\pm}}}^2.                                                      \label{Jj23}
} 
Collecting \eqref{Jj21}, \eqref{Jj22} and \eqref{Jj23}, we have 
\EQS{
 J_2 \lec \|n\|_{\dot{Y}^{s_c}_{W_{\pm c}}}^2 \|v\|_{Y^{s_c}_{K_{\pm}}}^2.    \label{j2}
} 
By Corollary \ref{U2A} and the triangle inequality to have 
\EQS{
 J_3 &\lec \sum_{N_1 \ge 1} N_1^{2s_c} \sup_{\|u\|_{V^2_{K_{\pm}}}=1} \Bigl{|}\sum_{N_2 \gg N_1} \sum_{N_3 \sim N_2} 
               \int_{\R^{1+d}} \til{n}_{N_3} (\om_1^{-1}\til{v}_{N_2})\ol{\til{u}_{N_1}} dxdt\Bigr{|}^2 \notag \\ 
      &\lec \sum_{N_1 \ge 1} N_1^{2s_c} \Bigl{(}\sum_{N_2 \gg N_1} \sum_{N_3 \sim N_2} \sup_{\|u\|_{V^2_{K_{\pm}}}=1} \Bigl{|}
               \int_{\R^{1+d}} \til{n}_{N_3} (\om_1^{-1}\til{v}_{N_2})\ol{\til{u}_{N_1}} dxdt\Bigr{|}\Bigr{)}^2.                   \label{j3a}
}
By the same manner as the estimate for Lemma \ref{tri} $(iii)$, we obtain 
\EQS{
 \Bigl{|}\int_{\R^{1+d}} \til{n}_{N_3} (\om_1^{-1}\til{v}_{N_2})\ol{\til{u}_{N_1}} dxdt\Bigr{|} 
    \lec N_3^{s_c} \|n_{N_3}\|_{V^2_{W_{\pm c}}} \|v_{N_2}\|_{V^2_{K_{\pm}}} \|u_{N_1}\|_{V^2_{K_{\pm}}}.   \label{j3d}
}
From \eqref{j3d}, the right-hand side of \eqref{j3a} is bounded by  
\EQQS{
 \sum_{N_1 \ge 1} \Bigl{(}\sum_{N_2 \gg N_1} \sum_{N_3 \sim N_2} N_1^{s_c} N_3^{s_c} \|n_{N_3}\|_{V^2_{W_{\pm c}}} 
       \|v_{N_2}\|_{V^2_{K_{\pm}}}\Bigr{)}^2. 
}
From $s_c > 0, \| \cdot \|_{l^2 l^1} \lec \| \cdot \|_{l^1 l^2}$ and the Cauchy-Schwarz inequality, we have 
\EQS{
 J_3^{1/2} &\lec \sum_{N_2 \gec 1} \sum_{N_3 \sim N_2} \Bigl{(}\sum_{N_1 \ll N_2} N_1^{2s_c} N_3^{2s_c} 
                        \|n_{N_3}\|_{V^2_{W_{\pm c}}}^2  \|v_{N_2}\|_{V^2_{K_{\pm}}}^2\Bigr{)}^{1/2} \notag \\
              &\lec \sum_{N_2 \gec 1} \sum_{N_3 \sim N_2} N_2^{s_c}N_3^{s_c}
                        \|n_{N_3}\|_{V^2_{W_{\pm c}}} \|v_{N_2}\|_{V^2_{K_{\pm}}} \notag \\ 
              &\lec \|n\|_{\dot{Y}^{s_c}_{W_{\pm c}}} \|v\|_{Y^{s_c}_{K_{\pm}}}.                   \label{j3}            
}
Collecting \eqref{j0}, \eqref{j1}, \eqref{j2} and \eqref{j3}, we obtain \eqref{BEKG}.  
We prove \eqref{BEW} below. 
By Corollary \ref{U2A}, we only need to estimate $K_i \ (i=1,2,3)$:  
\EQQS{
 &K_1 := \sum_{N_3} N_3^{2s_c} \sup_{\|n\|_{V^2_{W_{\pm c}}}=1} \Bigl{|}\sum_{N_2 \sim N_3} \sum_{N_1 \ll N_3} 
              \int_{\R^{1+d}} (\om_1^{-1}\til{u}_{N_1})(\ol{\om_1^{-1}\til{v}_{N_2}})(\ol{\om \til{n}_{N_3}})dxdt\Bigr{|}^2,  \\
 &K_2 := \sum_{N_3} N_3^{2s_c} \sup_{\|n\|_{V^2_{W_{\pm c}}}=1} \Bigl{|}\sum_{N_2 \ll N_3} \sum_{N_1 \sim N_3} 
              \int_{\R^{1+d}} (\om_1^{-1}\til{u}_{N_1})(\ol{\om_1^{-1}\til{v}_{N_2}})(\ol{\om \til{n}_{N_3}})dxdt\Bigr{|}^2,  \\
 &K_3 := \sum_{N_3} N_3^{2s_c} \sup_{\|n\|_{V^2_{W_{\pm c}}}=1} \Bigl{|}\sum_{N_2 \gec N_3} \sum_{N_1 \sim N_2} 
              \int_{\R^{1+d}} (\om_1^{-1}\til{u}_{N_1})(\ol{\om_1^{-1}\til{v}_{N_2}})(\ol{\om \til{n}_{N_3}})dxdt\Bigr{|}^2.
}
First, we estimate $K_1$. 
Put $K_1 = K_{1,1} + K_{1,2}$ where 
\EQS{
 &K_{1,1} := \sum_{N_3 \lec 1} N_3^{2s_c} \sup_{\|n\|_{V^2_{W_{\pm c}}}=1} \Bigl{|}\sum_{N_2 \sim N_3} \sum_{N_1 \ll N_3} 
                \int_{\R^{1+d}} (\om_1^{-1}\til{u}_{N_1})(\ol{\om_1^{-1}\til{v}_{N_2}}) \notag \\
 &\qquad \qquad \qquad \qquad \qquad \qquad \qquad \cross (\ol{\om \til{n}_{N_3}}) dxdt\Bigr{|}^2, 
                                                                                                              \label{K11} \\
 &K_{1,2} := \sum_{N_3 \gg 1} N_3^{2s_c} \sup_{\|n\|_{V^2_{W_{\pm c}}}=1} \Bigl{|}\sum_{N_2 \sim N_3} \sum_{N_1 \ll N_3} 
                \int_{\R^{1+d}} (\om_1^{-1}\til{u}_{N_1})(\ol{\om_1^{-1}\til{v}_{N_2}}) (\ol{\om \til{n}_{N_3}}) dxdt\Bigr{|}^2. 
                                                                                                                                           \notag   
}
By the same manner as the proof for Lemma \eqref{tri} $(i)$, we see  
\EQS{
 &\Bigl{|}\int_{\R^{1+d}} \Bigl{(}\sum_{N_1 \ll N_3}\om_1^{-1}\til{u}_{N_1}\Bigr{)} (\ol{\om_1^{-1}\til{v}_{N_2}})
            (\ol{\om \til{n}_{N_3}}) dxdt\Bigr{|} \notag \\
 &\lec \LR{N_2}^{-1/2}\LR{N_3}^{3/2}\|u\|_{Y^{s_c}_{K_{\pm}}} \|v_{N_2}\|_{V^2_{K_{\pm}}} \|n_{N_3}\|_{V^2_{W_{\pm c}}}.    \label{k1c} 
}  
Collecting \eqref{K11}, \eqref{k1c} and $N_2 \sim N_3 \lec 1$, we obtain 
\EQQS{
 K_{1,1} &\lec \sum_{N_2 \lec 1} N_2^{2s_c} (\|u\|_{Y^{s_c}_{K_{\pm}}} \LR{N_2}^{-1/2+3/2}\|v_{N_2}\|_{V^2_{K_{\pm}}})^2 \\
           &\lec \|u\|_{Y^{s_c}_{K_{\pm}}}^2 \sum_{N_2 \lec 1} N_2^{2s_c}\|v_{N_2}\|_{V^2_{K_{\pm}}}^2 \\
           &\lec \|u\|_{Y^{s_c}_{K_{\pm}}}^2 \|v\|_{Y^{s_c}_{K_{\pm}}}^2. 
}
For the estimate for $K_{1,2}$, we take $M=\e N_2$ for sufficiently small $\e >0$. 
Then, from Lemma \ref{Recovery}, we have
\EQQS{
 &P_{N_1} Q_{<M}^{K_{\pm}} \om_1^{-1} \bigl((Q_{<M}^{K_{\pm}} \om_1^{-1}\til{v}_{N_2})(Q_{<M}^{W_{\pm c}}{\om \, \til{n}_{N_3}})\bigr) \\
 &= P_{N_1} Q_{<M}^{K_{\pm}} \om_1^{-1} \Bigl[ \F^{-1}\Bigl( \int_{\ta_1 = \ta_2 + \ta_3,\, \xi_1 = \xi_2 + \xi_3} 
       \widehat{(Q_{<M}^{K_{\pm}} \om_1^{-1}\til{v}_{N_2})}(\ta_2, \xi_2)
       \widehat{(Q_{<M}^{W_{\pm c}} \om \, \til{n}_{N_3})}(\ta_3, \xi_3)  \Bigr) \Bigr] \\
 &=0
}
when $N_2 \gg \LR{N_1}$. 
Therefore, 
\EQQS{
 P_{N_1}\bigl( (\om_1^{-1}\til{v}_{N_2}) (\om \, \til{n}_{N_3})\bigr) = \sum_{i=1}^3 P_{N_1} G_i, 
}
where 
\EQQS{
 &G_1 :=  Q_{\ge M}^{K_{\pm}} \bigl((Q_2 \om_1^{-1}\til{v}_{N_2})(Q_3 \om \, \til{n}_{N_3})\bigr), \qquad 
   G_2 :=  Q_1 \bigl( (Q_{\ge M}^{K_{\pm}} \om_1^{-1}\til{v}_{N_2})(Q_3 \om \, \til{n}_{N_3}) \bigr), \notag \\
 &G_3 :=  Q_1 \bigl( (Q_2 \om_1^{-1}\til{v}_{N_2})(Q_{\ge M}^{W_{\pm c}} \om \, \til{n}_{N_3})\bigr).  
}
Here, $Q_1, Q_2 \in \{ Q_{<M}^{K_{\pm}}, Q_{\ge M}^{K_{\pm}} \}$ and $Q_3 \in \{ Q_{<M}^{W_{\pm c}}, Q_{\ge M}^{W_{\pm c}} \}$.  
Hence, it follows that 
\EQQS{
 K_{1,2} \le \sum_{i =1}^3 K_{1,2,i}
}
where  
\EQQS{
 K_{1,2,i} &:= \sum_{N_3 \gg 1} N_3^{2s_c} \sup_{\|n\|_{V^2_{W_{\pm c}}}=1} \Bigl{|}\sum_{N_2 \sim N_3}       
                        \sum_{N_1 \ll N_3} \int_{\R^{1+d}} (\om_1^{-1}\til{u}_{N_1}) \ol{G_i} dxdt\Bigr{|}^2, \qquad i=1,2,3. 
} 
By Lemma \ref{Mhigh}, we have 
\EQS{
 K_{1,2,1} &\lec  \sum_{N_3 \gg 1} N_3^{2s_c} \sup_{\|n\|_{V^2_{W_{\pm c}}}=1} \Bigl{|}\sum_{N_2 \sim N_3}
                        \sum_{N_1 \ll N_3} \int_{\R^{1+d}} (Q_{\ge M}^{K_{\pm}}\om_1^{-1}\til{u}_{N_1})(\ol{Q_2 \om_1^{-1}\til{v}_{N_2}}) 
                           \notag \\
            &\qquad \qquad       \cross (\ol{Q_3 \om \til{n}_{N_3}}) dxdt\Bigr{|}^2,                                       \label{k121} \\ 
 K_{1,2,2} &\lec  \sum_{N_3 \gg 1} N_3^{2s_c} \sup_{\|n\|_{V^2_{W_{\pm c}}}=1} \Bigl{|}\sum_{N_2 \sim N_3}
                        \sum_{N_1 \ll N_3} \int_{\R^{1+d}} (Q_1 \om_1^{-1}\til{u}_{N_1})(\ol{Q_{\ge M}^{K_{\pm}} \om_1^{-1}\til{v}_{N_2}}) 
                           \notag \\
            &\qquad \qquad       \cross (\ol{Q_3 \om \til{n}_{N_3}}) dxdt\Bigr{|}^2,                                       \label{k122} \\ 
 K_{1,2,3} &\lec  \sum_{N_3 \gg 1} N_3^{2s_c} \sup_{\|n\|_{V^2_{W_{\pm c}}}=1} \Bigl{|}\sum_{N_2 \sim N_3}
                        \sum_{N_1 \ll N_3} \int_{\R^{1+d}} (Q_1 \om_1^{-1}\til{u}_{N_1})(\ol{Q_2 \om_1^{-1}\til{v}_{N_2}}) 
                           \notag \\
            &\qquad \qquad       \cross (\ol{Q_{\ge M}^{W_{\pm c}} \om \til{n}_{N_3}}) dxdt\Bigr{|}^2.                                   \label{k123}  
}
By the same manner as the estimate for $F_2$, we apply Proposition \ref{bi-str}, $N_1 \ll N_2 \sim N_3, N_3 \gg 1$ and Proposition \ref{modulation}, then we obtain   
\EQS{
 &\Bigl{|} \int_{\R^{1+d}} (Q_{\ge M}^{K_{\pm}} \om_1^{-1} \til{u}_{N_1})(\ol{Q_2  \om_1^{-1} \til{v}_{N_2}})
     (\ol{Q_3 \om \til{n}_{N_3}})dxdt\Bigr{|} \notag \\ 
 &\lec \| Q_{\ge M}^{K_{\pm}} \om_1^{-1} \til{u}_{N_1}\|_{L^2_{t,x}} \| P_{N_1}\bigl{(} (\ol{Q_2  \om_1^{-1} \til{v}_{N_2}}) (\ol{Q_3 \om \til{n}_{N_3}})\bigr{)}\|_{L^2_{t,x}} \notag \\
 &\lec N_3^{-1/2} \LR{N_1}^{-1}\|u_{N_1}\|_{V^2_{K_{\pm}}} N_1^{(d-1)/2}(N_3/N_1)^{\e} \LR{N_2}^{-1}\|v_{N_2}\|_{V^2_{K_{\pm}}} N_3 \|n_{N_3}\|_{V^2_{W_{\pm c}}} \notag \\
 &\lec N_1^{s_c}(N_1/N_3)^{1/2-\e}\LR{N_2}^{-1}N_3\|u_{N_1}\|_{V^2_{K_{\pm}}}\|v_{N_2}\|_{V^2_{K_{\pm}}} \|n_{N_3}\|_{V^2_{W_{\pm c}}}. \label{k121c}
}
From \eqref{k121}, \eqref{k121c}, $N_3 \gg 1, N_2 \sim N_3$ and the Cauchy-Schwarz inequality, we have  
\EQQS{
 K_{1,2,1} 
  &\lec \sum_{N_2 \gg 1} N_2^{2s_c}\Bigl( \sum_{N_1 \ll N_2} N_1^{s_c}\|u_{N_1}\|_{V^2_{K_{\pm}}}(N_1/N_2)^{1/2-\e} \LR{N_2}^{-1}N_2\|v_{N_2}\|_{V^2_{K_{\pm}}}\Bigr)^2   \\
  &\lec \|u\|_{Y^{s_c}_{K_{\pm}}}^2 \|v\|_{Y^{s_c}_{K_{\pm}}}^2.
}
By Lemma \ref{tri} $(iv), i=5$, we obtain 
\EQS{
 &\Bigl{|} \int_{\R^{1+d}} \Bigl{(}\sum_{N_1 \ll N_3}Q_1 \om_1^{-1} \til{u}_{N_1}\Bigr{)}(\ol{Q_{\ge M}^{K_{\pm}} \om_1^{-1} \til{v}_{N_2}})
     (\ol{Q_3 \om \til{n}_{N_3}})dxdt\Bigr{|} \notag \\ 
 &\lec \LR{N_2}^{-1}N_3\|u\|_{Y^{s_c}_{K_{\pm}}}\|v_{N_2}\|_{V^2_{K_{\pm}}}\|n_{N_3}\|_{V^2_{W_{\pm c}}}.                              \label{k122c}
}
From \eqref{k122}, \eqref{k122c}, $N_3 \gg 1$ and $N_2 \sim N_3$, we have 
\EQQS{
 K_{1,2,2} 
  \lec \sum_{N_2 \gg 1} N_2^{2s_c}(\|u\|_{Y^{s_c}_{K_{\pm}}} \|v_{N_2}\|_{V^2_{K_{\pm}}})^2  
  \lec \|u\|_{Y^{s_c}_{K_{\pm}}}^2 \|v\|_{Y^{s_c}_{K_{\pm}}}^2.
}
By Lemma \ref{tri} $(iv), i=4$, we obtain 
\EQS{
 &\Bigl{|} \int_{\R^{1+d}} \Bigl{(}\sum_{N_1 \ll N_3}Q_1 \om_1^{-1} \til{u}_{N_1}\Bigr{)}(\ol{Q_2 \om_1^{-1} \til{v}_{N_2}})
     (\ol{Q_{\ge M}^{W_{\pm c}} \om \til{n}_{N_3}})dxdt\Bigr{|} \notag \\ 
 &\lec \LR{N_2}^{-1}N_3\|u\|_{Y^{s_c}_{K_{\pm}}}\|v_{N_2}\|_{V^2_{K_{\pm}}}\|n_{N_3}\|_{V^2_{W_{\pm c}}}.                              \label{k123c}
}
From \eqref{k123}, \eqref{k123c}, $N_3 \gg 1$ and $N_2 \sim N_3$, we have 
\EQQS{
 K_{1,2,3} 
  \lec \sum_{N_2 \gg 1} N_2^{2s_c}(\|u\|_{Y^{s_c}_{K_{\pm}}} \|v_{N_2}\|_{V^2_{K_{\pm}}})^2  
  \lec \|u\|_{Y^{s_c}_{K_{\pm}}}^2 \|v\|_{Y^{s_c}_{K_{\pm}}}^2.
}
By symmetry, the estimate for $K_2$ is obtained by the same manner as the estimate for $K_1$. Hence, we omit the estimate for $K_2$. 
By the triangle inequality, Lemma \ref{tri} $(i)$ and the Cauchy-Schwarz inequality, we have  
\EQQS{
 K_3^{1/2} &\lec \sum_{N_2} \sum_{N_1 \sim N_2}\Bigl{\{} \sum_{N_3 \lec N_2} N_3^{2s_c} \sup_{\|n\|_{V^2_{W_{\pm c}}}=1}
                  \Bigl{|}\int_{\R^{1+d}} (\om_1^{-1}\til{u}_{N_1})(\ol{\om_1^{-1}\til{v}_{N_2}})
                     (\ol{\om \til{n}_{N_3}})dxdt\Bigr{|}^2 \Bigr{\}}^{1/2} \\                                                                                                   
  &\lec \sum_{N_2} \sum_{N_1 \sim N_2} \Bigl{\{} \sum_{N_3 \lec N_2} N_3^{2s_c} (N_3^{s_c}
            \|u_{N_1}\|_{V^2_{K_{\pm}}} \|v_{N_2}\|_{V^2_{K_{\pm}}})^2 \Bigr{\}}^{1/2}  \\
  &\lec \sum_{N_2}\sum_{N_1 \sim N_2} N_1^{s_c} N_2^{s_c} \|u_{N_1}\|_{V^2_{K_{\pm}}}\|v_{N_2}\|_{V^2_{K_{\pm}}} \\
  &\lec \|u\|_{Y^{s_c}_{K_{\pm}}}\|v\|_{Y^{s_c}_{K_{\pm}}}.        
}
Therefore, we obtain \eqref{BEW}.     
\end{proof}
\section{The proof of the main theorem}
We define 
\EQQS{
 u_{\pm} := \om_1 u \pm i\p_t u, \quad  
 n_{\pm} := n \pm i (c\om)^{-1} \p_t n
}
where $\om_1 := (1-\laplacian )^{1/2}, \om := (-\laplacian )^{1/2}$.  
Then the wave equation in \eqref{KGZ} is rewritten into
\EQS{
 \begin{cases}
  i\p_t u_{\pm} \mp \om_1 u_{\pm} 
    = \pm (1/4)(n_+ + n_-)(\om_1^{-1}u_+ + \om_1^{-1}u_-), 
               \qquad (t,x) \in [-T,T] \cross \R^d, \\
  i\p_t n_{\pm} \mp c\om n_{\pm} 
    = \pm (4c)^{-1}\om | \om_1^{-1} u_+ + \om_1^{-1} u_-|^2, \qquad \qquad \quad (t,x) \in [-T,T] \cross \R^d, \\
  (u_{\pm}, n_{\pm})|_{t=0} = (u_{\pm 0}, n_{\pm 0}) 
                                \in H^s(\R^d) \cross \dot{H}^s(\R^d). 
                                                                                                                                            \label{KGZ''}
 \end{cases}  
}

Hence by the Duhamel principle, 
we consider the following integral equation corresponding to \eqref{KGZ''} on the time interval 
$[0, T)$ with $0< T \le \I:$
\EQ{ \label{int-eq}
  u_{\pm}=\Phi_1(u_{\pm},n_+,n_-), \quad  n_{\pm}=\Phi_2(n_{\pm},u_+,u_-), 
}
where
\EQQS{
  &\Phi_1(u_{\pm},n_+,n_-) := K_{\pm}(t)u_{\pm0} \pm (1/4)\{ I_{T,K_{\pm}}(n_+,u_+)(t) + I_{T,K_{\pm}}(n_+,u_-)(t) \\
                                 &\qquad \qquad \qquad \qquad  
                                      + I_{T,K_{\pm}}(n_-,u_+)(t) + I_{T,K_{\pm}}(n_-,u_-)(t) \}, \\
  &\Phi_2(n_{\pm},u_+,u_-) := W_{\pm c}(t)n_{\pm 0} \pm (4c)^{-1}\{ I_{T,W_{\pm c}}(u_+,u_+)(t) + I_{T,W_{\pm c}}(u_+,u_-)(t) \\
                                 &\qquad \qquad \qquad \qquad 
                                      + I_{T,W_{\pm c}}(u_-,u_+)(t) + I_{T,W_{\pm c}}(u_-,u_-)(t)\}.
}
\begin{prop}\label{main_prop1}
 (i) Let $d \ge 5, s = s_c = d/2-2$ and $\de > 0$ be sufficiently small.  
     For all $(u_{\pm0}, n_{\pm 0}) \in B_{\de }(H^s(\R^d) \cross \dot{H}^s(\R^d))$ and for all $0 < T < \I$,   
     there exists a unique solution of \eqref{int-eq} on $[0,T]$ such that 
\EQQ{
 (u_{\pm}, n_{\pm}) \in Y^s_{K_{\pm}}([0, T]) \cross \dot{Y}^s_{W_{\pm c}}([0, T])
      \subset C([0, T]; H^s(\R^d)) \cross C([0, T]; \dot{H}^s(\R^d)). 
} 
 (ii) The flow map obtained by (i):\\
    $B_{\de }(H^s(\R^d)) \cross B_{\de }(\dot{H}^s(\R^d)) 
 \ni (u_{\pm 0}, n_{\pm 0}) \mapsto
           (u_{\pm}, n_{\pm}) \in Y^s_{K_{\pm}}([0, T]) \cross \dot{Y}^s_{W_{\pm c}}([0, T])$
    is Lipschitz continuous.  
\end{prop}
\begin{rem} \label{time_rev}
Due to the time reversibility of the Klein-Gordon-Zakharov equation, 
Porpositions \ref{main_prop1} also holds in corresponding time interval $[-T,0]$
\end{rem} 
\begin{rem} \label{gsol}
By $(i)$ in Proposition \ref{main_prop1} and Remark \ref{time_rev}, for any $T>0$, we have solutions to 
\eqref{int-eq} $(u_{\pm}(t),n_{\pm}(t))$ on $[0,T]$ and $[-T,0]$. 
If initial data $(u_{\pm0}, n_{\pm 0}) \in B_{\de }(H^s(\R^d) \cross \dot{H}^s(\R^d))$, then we can take $T$ arbitrary large and by uniqueness, $(u_{\pm}(t),n_{\pm}(t)) \in C((-\I,\I);H^s(\R^d)) \cross C((-\I,\I);\dot{H}^s(\R^d))$ can be defined uniquely. 
\end{rem} 
\begin{prop} \label{main_prop2}
Let the solution $(u_{\pm}(t), n_{\pm}(t))$ to \eqref{int-eq} on $(-\I,\I)$ obtained by Proposirion \ref{main_prop1}, Remark \ref{time_rev} and Remark \ref{gsol} with initial data $(u_{\pm 0}, n_{\pm 0}) \in B_{\de}(H^s(\R^d) \cross \dot{H}^s(\R^d))$. 
Then, there exist $(u_{\pm, +\I}, n_{\pm, +\I})$ and $(u_{\pm, -\I}, n_{\pm, -\I})$ in $H^s(\R^d) \cross \dot{H}^s(\R^d)$ 
such that 
\EQQS{
 &\lim_{t \to +\I}(\|u_{\pm}(t)-K_{\pm}(t)u_{\pm, +\I}\|_{H^s_x(\R^d)} + \|n_{\pm}(t)-W_{\pm c}(t)n_{\pm, +\I}\|_{\dot{H}^s_x(\R^d)}) = 0, \\
 &\lim_{t \to -\I}(\|u_{\pm}(t)-K_{\pm}(t)u_{\pm, -\I}\|_{H^s_x(\R^d)} + \|n_{\pm}(t)-W_{\pm c}(t)n_{\pm, -\I}\|_{\dot{H}^s_x(\R^d)}) = 0.
}   
\end{prop}

\begin{proof}[proof of Proposition \ref{main_prop1}]
First, we prove $(i)$. 
By Proposition \ref{Str}, there exists $C>0$ such that   
\EQQS{
 \|K_{\pm}(t)u_{\pm0}\|_{Y^s_{K_{\pm}}} \le C\|u_{\pm0}\|_{H^s}, \qquad 
 \|W_{\pm c}(t)n_{\pm 0}\|_{\dot{Y}^s_{W_{\pm c}}} \le C\|n_{\pm 0}\|_{\dot{H}^s}.
} 
We denote time interval $I := [0, T]$. 
If $(u_{\pm 0}, n_{\pm 0}) \in B_{\de}(H^s(\R^d) \cross \dot{H}^s(\R^d))$ is small and $(u_{\pm}, n_{\pm}) \in B_r(Y^s_{K_{\pm}}(I) \cross \dot{Y}^s_{W_{\pm c}}(I)), s = d/2-2$, then by Proposition \ref{BE} and Remark \ref{remconstbe}, we have 
\EQQS{
 &\|\Phi_1(u_{\pm}, n_+,n_-)\|_{Y^s_{K_{\pm}}(I)} \\ 
  &\le C\de + (C/4)(\|n_+\|_{\dot{Y}^s_{W_{+c}}(I)}\|u_+\|_{Y^s_{K_+}(I)} 
         + \|n_+\|_{\dot{Y}^s_{W_{+c}}(I)}\|u_-\|_{Y^s_{K_-}(I)} \\
  &\qquad 
         + \|n_-\|_{\dot{Y}^s_{W_{-c}}(I)}\|u_+\|_{Y^s_{K_+}(I)} 
         + \|n_-\|_{\dot{Y}^s_{W_{-c}}(I)}\|u_-\|_{Y^s_{K_-}(I)}), \\
 &\|\Phi_2(n_{\pm},u_+,u_-)\|_{\dot{Y}^s_{W_{\pm c}}(I)} \\
  &\le C\de + (C/4c)(\|u_+\|^2_{Y^s_{K_+}(I)} + 2\|u_+\|_{Y^s_{K_+}(I)}\|u_-\|_{Y^s_{K_-}(I)} 
         + \|u_-\|_{Y^s_{K_-}(I)}^2). 
} 
Taking $\de = r^2$ and $r=\min\{1,c\}/(4C)$, then we have 
\EQQS{
 \|\Phi_1(u_{\pm}, n_+,n_-)\|_{Y^s_{K_{\pm}}(I)} \le r, \qquad 
 \|\Phi_2(n_{\pm},u_+,u_-)\|_{\dot{Y}^s_{W_{\pm c}}(I)} \le r. 
} 
Hence, $(\Phi_1, \Phi_2)$ is a map from $B_r(Y^s_{K_{\pm}}([0,T]) \cross \dot{Y}^s_{W_{\pm c}}([0,T]))$ into itself.  
If we also assume $(v_{\pm}, m_{\pm}) \in B_r(Y^s_{K_{\pm}}(I) \cross \dot{Y}^s_{W_{\pm c}}(I))$, then we have 
\EQS{
 &\|\Phi_1(u_{\pm}, n_+, n_-)-\Phi_1(v_{\pm}, m_+,m_-)\|_{Y^s_{K_{\pm}}(I)} \notag \\
 &\le (1/8)(\|u_+ - v_+\|_{Y^s_{K_+}(I)} + \|u_- - v_-\|_{Y^s_{K_-}(I)}  \notag   \\
   &\qquad + \|n_+ - m_+\|_{\dot{Y}^s_{W_{+c}}(I)} + \|n_- - m_-\|_{\dot{Y}^s_{W_{-c}}(I)}), \label{kg-contraction} \\
 &\|\Phi_2 (n_{\pm},u_+,u-) - \Phi_2 (m_{\pm}, v_+, v_-)\|_{\dot{Y}^s_{W_{\pm c}}(I)}  \notag  \\
  &\le (1/4)(\|u_+ - v_+\|_{Y^s_{K_+}(I)} + \|u_- - v_-\|_{Y^s_{K_-}(I)}).  \label{w-contraction}
}
Thus, $(\Phi_1, \Phi_2)$ is a contraction mapping on $B_r(Y^s_{K_{\pm}}([0,T]) \cross \dot{Y}^s_{W_{\pm c}}([0,T]))$.  
Hence, by the Banach fixed point theorem, we have a solution to \eqref{int-eq} in it.   
We assume that $(u_{\pm}(0), n_{\pm}(0)), (v_{\pm}(0), m_{\pm}(0))$ are both small and $s = d/2-2$ for $d \ge 5$.  
Let $(u_{\pm}, n_{\pm}), (v_{\pm}, m_{\pm}) \in Y^s_{K_{\pm}}([0,T]) \cross \dot{Y}^s_{W_{\pm c}}([0,T])$ are two solutions satisfying $(u_{\pm}(0), n_{\pm}(0))=(v_{\pm}(0), m_{\pm}(0))$. Moreover, 
\EQQS{
 T' := \sup \{0 \le t \le T \, ; u_{\pm}(t)=v_{\pm}(t), n_{\pm}(t)=m_{\pm}(t) \} <T. 
}
By a translation in $t$, it suffices to consider $T'=0$. Let $0 < \ta \le T$ be fixed later. From \eqref{kg-contraction}--\eqref{w-contraction} and Proposition \ref{unique}, we obtain 
\EQS{ 
 &\|u_{\pm} - v_{\pm}\|_{Y^s_{K_{\pm}}([0,\ta ])}  \notag \\
  &\le (1/7)(\|n_+ -m_+\|_{\dot{Y}^s_{W_{+c}}([0,\ta ])} + \|n_- - m_-\|_{\dot{Y}^s_{W_{-c}}([0,\ta ])} 
                + \|u_{\mp} - v_{\mp}\|_{Y^s_{K_{\mp}}([0,\ta ])}),      \label{u-v}  \\
 &\|n_{\pm} - m_{\pm}\|_{\dot{Y}^s_{W_{\pm c}}([0,\ta ])} 
   \le (1/4)(\|u_+ - v_+\|_{Y^s_{K_+}([0,\ta ])} + \|u_- - v_-\|_{Y^s_{K_-}([0,\ta ])}).      \label{n-m}
}
From \eqref{u-v} and \eqref{n-m}, we obtain 
\EQQS{
 u_{\pm}=v_{\pm}, \quad n_{\pm} = m_{\pm}
}
on $[0,\ta ]$ if $0 < \ta \le T$ be sufficiently small. This contradicts the definition of $T'$. Therefore, the uniqueness of the solution $(u_{\pm},n_{\pm})$ is showed. $(ii)$ follows from the standard argument, so we omit the proof. 
\end{proof}

Finally, we prove Proposition \ref{main_prop2}. The proof is the same manner as the proof for Proposition 4.2 in ~\cite{KaT}. 
\begin{proof}
There exists $M > 0$ such that for all $0 < T < \I$, 
\EQQS{
 &\|u_{\pm}\|_{Y^s_{K_{\pm}}([0,T])} + \|n_{\pm}\|_{\dot{Y}^s_{W_{\pm c}}([0,T])} < M, \\
 &\|u_{\pm}\|_{Y^s_{K_{\pm}}([-T,0])} + \|n_{\pm}\|_{\dot{Y}^s_{W_{\pm c}}([-T,0])} < M 
} 
holds since $r$ in the proof of Proposition \ref{main_prop1} does not depend on $T$. 
Take $\{t_k\}_{k=0}^K \in \mathcal{Z}_0$ and $0< T < \I$ such that $-T < t_0, t_K < T$. 
By $L^2_x$ orthogonality, 
\EQQS{
 &\Bigl( \sum_{k=1}^K \| \LR{\na_x}^s \bigl( K_{\pm}(-t_k)u_{\pm}(t_k) - 
        K_{\pm}(-t_{k-1})u_{\pm}(t_{k-1})\bigr)\|_{L^2_x}^2\Bigr)^{1/2} \\
 &\lec \| \LR{\na_x}^s u_{\pm}\|_{V^2_{K_{\pm}}([0,T])} + \| \LR{\na_x}^s u_{\pm}\|_{V^2_{K_{\pm}}([-T,0])} \\
 &\lec \|u_{\pm}\|_{Y^s_{K_{\pm}}([0,T])} + \|u_{\pm}\|_{Y^s_{K_{\pm}}([-T,0])} \\
 &< 2M.   
} 
Thus, 
\EQQS{
 \sup_{\{t_k\}_{k=0}^K \in \mathcal{Z}_0} \Bigl( \sum_{k=1}^K \| \LR{\na_x}^s K_{\pm}(-t_k)u_{\pm}(t_k) - 
        \LR{\na_x}^s K_{\pm}(-t_{k-1})u_{\pm}(t_{k-1}) \|_{L^2_x}^2\Bigr)^{1/2} 
 \lec M.
}
Hence, there exists $f_{\pm} := \lim_{t \to \pm \I} \LR{\na_x}^s K_{\pm}(-t)u_{\pm}(t)$ in $L^2_x(\R^d)$. 
Then put $u_{\pm \I} := \LR{\na_x}^{-s}f_{\pm}$, we obtain 
\EQQS{
 \| \LR{\na_x}^s K_{\pm}(-t)u_{\pm}(t)-f_{\pm}\|_{L^2_x} 
  = \| u_{\pm}(t) - K_{\pm}(t)u_{\pm \I}\|_{H^s_x} 
  \to 0  
} 
as $t \to \pm \I$.
The scattering result for the wave equation is obtained similarly.  

\end{proof}

\end{document}